\documentclass[leqno]{article}
\usepackage{amsmath}
\usepackage{amssymb}
\usepackage{amsthm}
\usepackage[dvipdfmx]{graphicx}
\usepackage{url}
\usepackage{pxfonts}
\usepackage{bm}
\title{\bf Rigorous numerical computations for 1D advection equations with variable coefficients}
\author{
Akitoshi Takayasu\thanks{Faculty of Engineering, Information and Systems, University of Tsukuba, 1-1-1 Tennodai, Tsukuba, Ibaraki 305-8573, Japan (\texttt{takitoshi@risk.tsukuba.ac.jp})},~
Suro Yoon\thanks{Department of Risk Engineering, University of Tsukuba (\texttt{s1620591@u.tsukuba.ac.jp}).},~
and Yasunori Endo\thanks{Faculty of Engineering, Information and Systems, University of Tsukuba (\texttt{endo@risk.tsukuba.ac.jp}).}
}
\date{\today}
\newtheorem{theorem}{Theorem}[section]    % Standard theorem environment
\newtheorem{lemma}[theorem]{Lemma}          % Lemma environment with numbering 
\theoremstyle{definition}
\newtheorem{definition}[theorem]{Definition}    % Definition environment with 
\newtheorem{remark}{Remark}             % Unnumbered environment for remarks.
\begin{document}
\maketitle
\begin{abstract}
This paper provides a methodology of verified computing for solutions to 1-dimensional advection equations with variable coefficients.
The advection equation is typical partial differential equations (PDEs) of hyperbolic type.
There are few results of verified numerical computations to initial-boundary value problems of hyperbolic PDEs.
Our methodology is based on the spectral method and semigroup theory.
The provided method in this paper is regarded as an efficient application of semigroup theory in a sequence space associated with the Fourier series of unknown functions.
This is a foundational approach of verified numerical computations for hyperbolic PDEs.
Numerical examples show that the rigorous error estimate showing the well-posedness of the exact solution is given with high accuracy and high speed.
\end{abstract}
%%%%
%%%%%%%%%%%%%%%%%%%%%%%%%%%%%%%%%%%%%%%%%%%%%%

%\tableofcontents

\section{Introduction}
% Problem, Background
Let $\mathbb{R}$ be the set of real numbers and let $\Omega=(0,2\pi)\subset\mathbb{R}$.
In this paper, we consider smooth solutions of the following initial-boundary value problem of 1-dimensional advection equations with variable coefficients: 
\begin{align}\label{eqn:advec_Eq}
\begin{cases}
u_t+c(x)u_x=0,~x\in \Omega,~t>0,\\
u(0,t)=u(2\pi,t),~t>0,\\
u(x,0)=u_0(x),~x\in\Omega,
\end{cases}
\end{align}
where $(x,t)$ are the space-time variables.
Partial derivatives of $u(x,t)$ are denoted by $u_t:=\frac{\partial}{\partial t}u(x,t)$ and $u_x:=\frac{\partial}{\partial x}u(x,t)$.
The variable coefficient $c(x)$ is a space-dependent positive real-valued function satisfying the periodic boundary condition.
We also require that the initial function $u_0(x)$ is a periodic function.
It is well-known (cf., e.g., \cite{bib:Evans1998}) that the advection equation \eqref{eqn:advec_Eq} is a typical partial differential equation (PDE) of hyperbolic type.
It appears in mathematical models of the traffic stream (nonlinear waves), compressive fluids, and systems of conservation law, etc.
%手法の応用が期待できる。
The well-posedness of \eqref{eqn:advec_Eq} is shown on a suitable function space \cite{bib:hesthaven2007}.
It is also known (cf.\,\cite{bib:Evans1998}) that behavior of the solution follows the \emph{characteristic curve}.
For example, if $\Omega=\mathbb{R}$ and a curve $x\equiv X(t)$ defined on $x$--$t$ upper half-plane satisfies
\[
	\frac{dX}{dt}=c\left(X(t)\right),
\]
then the solution of \eqref{eqn:advec_Eq} without the boundary condition is expressed by $u(x,t)=u_0\left(X(t)\right)$.
This means that the solution shows a flow of initial distribution.
Moreover, if two characteristic curves cross, a singular solution so-called \emph{shock wave} occurs.
In such a case, regularity of the solution is lost any more and one needs to discuss weak solutions in the distributional sense.

% Aim and difficulties: verified computations
Under the above analytic background, the aim of this paper is to figure out behavior of smooth solutions of \eqref{eqn:advec_Eq} rigorously by using numerical computations.
Basically, numerical computations may help us to understand the solution quantitatively but there include several kind of errors for computing solutions numerically.
What is worse, many difficulties appear in numerical computing for the solution  of \eqref{eqn:advec_Eq}.
Numerical computations often become unstable.
The dissipation of numerical solutions is sometimes caused by the discretization of PDEs, which is so-called \emph{numerical dispersibility}.
The propagation of velocity is also changed by the discretization.
Furthermore, in the viewpoint of \emph{verified numerical computations}, there are few results \cite{bib:Minamoto1997,bib:Minamoto2001,bib:Minamoto2001a,bib:Nakao1994} for initial-boundary value problems of hyperbolic PDEs.
Verified numerical computations for PDEs have been established by Nakao \cite{bib:Nakao1988} and Plum \cite{bib:Plum1991} independently.
These have been developed in the last three decades by their collaborators and many researchers in the field of dynamical systems (see, e.g., \cite{bib:Figueras2016,bib:Mizuguchi2014,bib:Nakao2014,bib:Nakao2011,bib:plum_survey,bib:Takayasu2013,bib:Nobito1998,bib:Zgliczynski2001} and references therein).
Recently, verified numerical computations enable us to understand \emph{traveling-waves, periodic solutions, invariant objects} (including \emph{stationary solutions}) of parabolic/elliptic PDEs, etc.
For such a research field, it is a challenging task for introducing a methodology of verified numerical computations for hyperbolic PDEs.

% Our tools: Semigroup
One of our main tools in this paper is \emph{semigroup} in the classical analysis.
For initial-boundary value problems of parabolic PDEs, the first author and his collaborators have introduced a methodology of verified numerical computations \cite{bib:Mizuguchi2014,bib:Mizuguchi2017,bib:Takayasu2017} using semigroup theory.
Such a method is based on the \emph{analytic semigroup} generated by the Laplacian $\Delta$.
The semigroup is a solution operator of the Cauchy problem.
For example, a solution of the heat equation
\[
	u_t=\Delta u~\mbox{(with boundary condition)},~u(x,0)=u_0(x)
\]
is expressed by $u(x,t)=e^{\Delta t}u_0(x)$ using the analytic semigroup $\{e^{\Delta t}\}_{t\ge0}$ on a certain Banach space $X$.
These studies \cite{bib:Mizuguchi2014,bib:Mizuguchi2017,bib:Takayasu2017} have achieved an efficient combination of \emph{operator theory} and verified numerical computations.
Such a combination can be expected to provide an approach to many unsolved problems of PDEs.

% Our tools: Spectral method
Another tool of the present paper is the \emph{spectral method} (cf. \cite{bib:boyd2013,bib:hesthaven2007,bib:Trefethen2000}) for numerically computing the solution of \eqref{eqn:advec_Eq}.
The most advantage of using the spectral method is accuracy of numerical solutions.
If a solution is smooth enough, a numerically computed approximate solution is much more accurate than that by other numerical methods, e.g., the finite difference method, the finite element method, etc.
On the other hand, the spectral method restricts boundary conditions of problems and the shape of domain $\Omega$.
It is moreover difficult to deal with the convolution term for nonlinear problems.
Although flawed, accurate numerical approximate solution is useful for verified computing to PDEs.
In particular, the decay property of coefficients of series expansion gives great benefit to verified numerical computations.
By making good use of such a property, verification methods for analytic solutions to PDEs have been proposed in \cite{bib:Figueras2016,bib:Hungria2016,bib:Lessard2017}.
In our study, inspired by these studies, we construct a solution using the Fourier series with respect to $x$-variable and the Chebyshev series with respect to $t$-variable.
These series expansions are appropriate for the initial-boundary value problem \eqref{eqn:advec_Eq}.

% Our contribution
The main contribution of this paper is to provide a method of verified computing by using sequence spaces, which come from the Fourier series of the unknown function.
Our methodology is based on the spectral method \cite{bib:boyd2013,bib:hesthaven2007,bib:Trefethen2000}, i.e., we try to enclose the Fourier coefficients of the exact solution of \eqref{eqn:advec_Eq}.
Let $\mathbb{Z}$ be a set of integers and let $i=\sqrt{-1}$ denotes the imaginary unit.
Considering the Fourier series of the unknown function and that of the variable coefficient in \eqref{eqn:advec_Eq}
\[
	u(x,t)=\sum_{k\in\mathbb{Z}}a_k(t)e^{ikx},\quad c(x)=\sum_{k\in\mathbb{Z}}c_ke^{ikx},
\]
the advection equation \eqref{eqn:advec_Eq} is equivalent to the following infinite-dimensional system of ordinary differential equations (ODEs):
\begin{align}\label{eqn:infODEs}
\frac{d}{dt}a_k(t)+\sum_{m\in\mathbb{Z}}c_{k-m}ima_m(t)=0,~k\in\mathbb{Z}.
\end{align}
Let $a(t):=(a_k(t))_{k\in\mathbb{Z}}$ for $t>0$ and $c:=(c_k)_{k\in\mathbb{Z}}$.
We consider \eqref{eqn:infODEs} in a certain sequence space $X$.
Then, \eqref{eqn:infODEs} can be represented by the $X$-valued Cauchy problem
\[
	\frac{d}{dt}a(t)-Aa(t)=0~\mbox{in}~X,
\]
where $A$ is an operator defined by
\begin{align}\label{eqn:opA}
	Aa(t):=-\left(\sum_{m\in\mathbb{Z}}c_{k-m}ima_m(t)\right)_{k\in\mathbb{Z}}.
\end{align}
We will prove that this operator $A$ generates the $C_0$ semigroup\footnote{The definition of $C_0$ semigroup is given in Section \ref{sec2}.} on $X$ under a certain assumption.
The $C_0$ semigroup is used to derive an error estimate between the exact solution and a numerically computed approximate solution.
Such a error estimate is our main result for verifying the behavior of the exact solution of \eqref{eqn:advec_Eq}.
This methodology is regarded as an efficient application of semigroup theory in the sequence space for verified numerical computations.

% Profile of the paper
The rest of this paper is organized as follows: 
In Section \ref{sec2}, we briefly introduce semigroup theory in the classical analysis.
A sufficient condition of the operator $A$ generating $C_0$ semigroup on the sequence space $X$ is given in Theorem \ref{thm:sg_generate}.
We also derive a norm estimate of the $C_0$ semigroup by using an operator norm of $X$.
In Section \ref{sec3}, we introduce a rigorous error estimate using the $C_0$ semigroup, which rigorously bounds the error between the Fourier coefficients of exact solution and those of numerically computed approximate solution in the sense of $X$ topology.
Theorem \ref{thm:main_thm} is our main theorem in the present paper, which shows a sufficient condition for proving the well-posedness of \eqref{eqn:advec_Eq} in analytic category for the space variable.
Subsequently, we give a detailed estimate of initial error and that of residual.
In Section \ref{sec:num_result}, we numerically demonstrate efficiency of the provided method of verified computing.
Finally, as a conclusion, we discuss potential applications of the provided method to mathematical models of nonlinear waves.

\section{Semigroup on a sequence space}\label{sec2}
%In this section, we show that the operator $A$ defined in \eqref{eqn:OpA} generates a {\em $C_0$ semigroup of $G(1,\omega)$ class}, whose definition will be given below.
\subsection{Preliminaries from semigroup theory}
We prepare some definitions and theorems associated with semigroup theory (for the details see, e.g., \cite{bib:Ito2002,bib:pazy,bib:yagi}, etc.).
\begin{definition}\label{def:semigroup}
Let $X$ be a Banach space.
A family of bounded linear operators $\{S(t)\}_{t\ge0}$ on $X$ is called a strongly continuous semigroup if
\begin{enumerate}
\item $S(0)=I$, where $I$ is the identity operator on $X$,
\item $S(t+s)=S(t)S(s)$, for any $t,s\ge0$,
\item $\lim\limits_{t\downarrow 0}\|S(t)-I\|_{X,X}=0$,
where $\|\cdot\|_{X,X}$ denotes the operator norm on $X$.
\end{enumerate}
\end{definition}
The strongly continuous semigroup will be called simply {\em$C_0$ semigroup}.
A well-known property of the $C_0$ semigroup is given by the following theorem:
\begin{theorem}[{\cite[Theorem 2.2]{bib:pazy}}]
Let $S(t)$ be the $C_0$ semigroup.
There exist constants $\omega\ge0$ and $M\ge1$ such that
\begin{align}\label{eqn:GMome}
	\|S(t)\|_{X,X}\le Me^{\omega t},~~0\le t<\infty.
\end{align}
\end{theorem}
%Here, the proof is omitted.
The $C_0$ semigroup satisfying \eqref{eqn:GMome} is called the {\em $C_0$ semigroup of $G(M,\omega)$ class} (cf.~\cite{bib:Ito2002}).
\begin{definition}
The linear operator $A$ defined by
\[
	Ax:=\lim\limits_{t\downarrow 0}\frac{S(t)x-x}{t},~~x\in D(A):=\left\{x\in X:\lim\limits_{t\downarrow 0}\frac{S(t)x-x}{t}~\mbox{exists}\right\}
\]
is called the infinitesimal generator of the semigroup $S(t)$.
Here, $D(A)$ is the domain of $A$.
\end{definition}

Let $X$ be a Banach space and let $X^\ast$ be its dual space.
The dual product between $X$ and $X^*$ is denoted by $\langle\cdot,\cdot\rangle_{X,X^\ast}$.
The norms of $X$ and $X^\ast$ are denoted by $\|\cdot\|_X$ and $\|\cdot\|_{X^\ast}:=\sup_{0\not=x\in X}\frac{\left|\langle x,\cdot\rangle_{X,X^\ast}\right|}{\|x\|_X}$, respectively.
From H\"older's inequality, $|\langle x,x^\ast\rangle_{X,X^\ast}|\le\|x\|_X\|x^\ast\|_{X^\ast}$ holds for $x\in X$ and $x^\ast\in X^\ast$.
For $x\in X$, we define the duality set\footnote{From the Hahn-Banach theorem (cf., e.g., \cite{bib:Brezis2010}), it follows that $F(x)\not=\emptyset$ for any $x\in X$.} as
\begin{align}\label{eqn:dualityset}
	F(x):=\left\{x^\ast\in X^\ast:\langle
	 x,x^\ast\rangle_{X,X^\ast}=\|x\|_{X}^2=\|x^\ast\|_{X^\ast}^2\right\}.
\end{align}
\begin{definition}
Let $\omega>0$.
A linear operator $A$ is $\omega$-dissipative if for any $x\in D(A)$ there exists $x^\ast\in F(x)$ such that
\[
	\operatorname{Re}\langle Ax,x^\ast\rangle_{X,X^\ast}\le \omega\|x\|_X^2.
\]
\end{definition}

The necessary and sufficient condition for $A$ generating the $C_0$ semigroup is shown by the Lumer-Phillips theorem \cite{bib:Lumer1961}, which is equivalent to the Hille-Yosida theorem \cite{bib:hille1948,bib:yosida1948}.
\begin{theorem}[Lumer-Phillips \cite{bib:Lumer1961}]\label{thm:LPtheorem}
Let $X$ be a Banach space and let $A$ be a closed linear operator with dense domain $D(A)$ in $X$.
The followings are equivalent:
\begin{enumerate}
\item The operator $A$ is the infinitesimal generator of the $C_0$ semigroup of $G(1,\omega)$ class.
\item The operator $A$ is $\omega$-dissipative and there exists $\lambda_0>\omega$ such that the range of $\lambda_0 I-A$ is X, i.e.,
\[
	R(\lambda_0 I-A)=X.
\]
\end{enumerate}
\end{theorem}
%The semigroup generated by $A$ gives the solution of the Cauchy problem in $X$.
\begin{remark}
The original Lumer-Phillips theorem is given for the $C_0$ semigroup of $G(1,0)$ class.
It is easy to extend the result for the semigroup of $G(1,\omega)$ class (cf. \cite{bib:pazy}).
Indeed, letting $S(t)$ be the $C_0$ semigroup of $G(1,\omega)$ class and $T(t)=e^{-\omega t}S(t)$, $T(t)$ is obviously the $C_0$ semigroup of $G(1,0)$ class.
Furthermore, if $A$ is the infinitesimal generator of $S(t)$ then $A - \omega I$ is the infinitesimal generator of $T(t)$.
On the other hand, if $\tilde{A}$ is the infinitesimal generator of $T(t)$, then $\tilde{A}+\omega I$ is the infinitesimal generator of $S(t)$ satisfying $\|S(t)\|_{X,X}\le e^{\omega t}$.
Thus, $S(t) = e^{\omega t}T(t)$.
It means that the original Lumer-Phillips theorem holds when one consider $\tilde{A}=A-\omega I$ instead of $A$.
\end{remark}

\subsection{Generation of $C_0$ semigroup on sequence spaces}\label{sec:generation_sg}
Let $\mathbb{C}$ be a set of complex numbers.
We define
\[
	\ell^p:=\left\{(a_k)_{k\in\mathbb{Z}}:\sum_{k\in\mathbb{Z}}|a_k|^p<\infty,~a_k\in\mathbb{C}\right\}\quad(1\le p<\infty)%~\mbox{and}~
%	\ell^\infty:=\left\{(a_k)_{k\in\mathbb{Z}}:\sup_{k\in\mathbb{Z}}|a_k|<\infty,~a_k\in\mathbb{C}\right\}.
\]
endowed with norms
\[
	\|a\|_p:=\left(\sum_{k\in\mathbb{Z}}|a_k|^p\right)^{1/p}\quad(1\le p<\infty).
\]
In the rest of the paper, we set the Banach space $X=\ell^2$ and its dual space $X^\ast=\ell^2\left(=\left(\ell^2\right)^\ast\right)$.
%\[
%	\|a\|_X:=\left(\sum_{k\in\mathbb{Z}}|a_k|^2\right)^{1/2},\quad\|a\|_{X^\ast}=\|a\|_X,%:=\sup_{k\in\mathbb{Z}}|a_k|,
%\]
%respectively.
%Let $a:=(a_k)_{k\in\mathbb{Z}}$ and $c:=(c_k)_{k\in\mathbb{Z}}$ be sequences of complex numbers.
Let us define the multiply operator $B$ as
\[
	Ba:=(ika_k)_{k\in\mathbb{Z}}.
\]
%where $i=\sqrt{-1}$ denotes the imaginary unit.
The domain of $B$ is denoted by $D(B):=\left\{a\in X:Ba\in X\right\}$.
%\footnote{%
	The domain $D(B)$ is dense in $X$ because, for a sequence $a\in D(B)$, the finite-dimensional subspace defined by
	\[
	\tilde{X}_N:=\left\{(\tilde{a}_k)_{k\in\mathbb{Z}}:\tilde{a}_k=a_k~(|k|\le N),~\tilde{a}_k=0~(|k|>N)\right\}\subset D(B)
	\]
	is dense in $X$, i.e., for any $\varepsilon>0$ there exists a positive integer $N$ such that one can obtain $\tilde{a}\in\tilde{X}_n$ ($n\ge N$) satisfying $\|a-\tilde{a}\|_X<\varepsilon$.
%}.
Let $c=(c_k)_{k\in\mathbb{Z}}$ be a sequence of complex numbers and let us assume $Bc\in \ell^1$.
As defined in \eqref{eqn:opA}, the operator $A$ is denoted by the discrete convolution\footnote{%
	The "$*$" denotes the discrete convolution (Cauchy product) defined by $a*b=\left(\sum_{m\in\mathbb{Z}}a_{k-m}b_m\right)_{k\in\mathbb{Z}}$ for two bi-infinite sequences $a$ and $b$. %If $a,b\in X$, then $\|a*b\|_X\le\|a\|_X\|b\|_X$ holds because the Banach space $X$ is \emph{Banach algebra} under the the discrete convolution.
} of $c$ and $Ba$, i.e.,
\begin{align*}
	Aa=-c*(Ba)=-\left(\sum_{m\in\mathbb{Z}}c_{k-m}ima_m\right)_{k\in\mathbb{Z}}.
\end{align*}
The domain of $A$ is denoted by $D(A):=\left\{a\in X:c*a\in D(B)\right\}$.
For the operator $A$ and its domain $D(A)$, we have the following two lemmas:
%\memo{$D(A)$のdensity, $A$のclosedness,}
\begin{lemma}\label{lem:dense}
$D(A)$ is dense in $X$.
\end{lemma}
\begin{proof}
It is sufficient to prove $D(B)\subset D(A)$.
For $a\in D(B)$, we have
\[
	k\left(c*a\right)_k=k\sum_{m\in\mathbb{Z}}c_{k-m}a_k=\sum_{m\in\mathbb{Z}}(k-m)c_{k-m}a_m+\sum_{m\in\mathbb{Z}}c_{k-m}ma_m,~k\in\mathbb{Z}.
\]
Then, Young's inequality yields 
%Summing this with respect to $k$, we have
\begin{align*}
	\left\| B\left(c*a\right)\right\|_X&\le \left\|Bc*a\right\|_X+\left\|c*Ba\right\|_X\\
	&\le\left\|Bc\right\|_1\left\|a\right\|_X+\left\|c\right\|_1\left\|Ba\right\|_X.
%%	\sum_{k\in\mathbb{Z}}|k|\left|\left(c*a\right)_k\right|\le\|Bc\|_1\|a\|_X+\|c\|_1\|Ba\|_X.
\end{align*}
This is bounded because $a\in D(B)$ and $Bc\in\ell^1$.
%The proof is completed.
\end{proof}
\begin{lemma}\label{lem:closedness}
The operator $A:D(A)\subset X\to X$ is closed.
\end{lemma}
\begin{proof}
Consider the sequence $\{a^n\}\subset D(A)$ satisfying $a^n\to a$ and $Aa^n\to b$ in $X$ (as $n\to\infty$).
We have $c*a^n\to c*a$ in $X$ because
\[
	\left\|c*a^n-c*a\right\|_X=\left\|c*(a^n-a)\right\|_X\le\|c\|_1\|a^n-a\|_X\to 0~(n\to\infty).
\]
From the elementary calculation,
\[
	B\left(c*\phi\right)=(Bc)*\phi+c*\left(B\phi\right)=(Bc)*\phi-A\phi
\]
holds for $\phi\in D(A)$.
Then, we have
\[
	B\left(c*a^n\right)=(Bc)*a^n-Aa^n\to (Bc)*a-b~\mbox{in}~X.
\]
The closedness of $B$ implies that $c*a\in D(B)$ and $B(c*a)=(Bc)*a-b$.
This yields $a\in D(A)$ and $Aa=b$.
\end{proof}

The above two lemmas show that the operator $A$ is densely defined on $X$.
The following theorem gives a sufficient condition for the operator $A$ generating the $C_0$ semigroup of $G(1,\omega)$ class:
\begin{theorem}\label{thm:sg_generate}
Let $Bc\in\ell^1$ and $c_{-m}=\overline{c_m}$ $(m\in\mathbb{Z})$, where $\overline{c_m}$ denotes the complex conjugate of $c_m$.
If the sequence $c$ also satisfies
\begin{align}\label{eqn:ass_c}
	|c_0|-2\sum_{m\not=0}|c_m|>0,
\end{align}
then the operator $A$ generates the $C_0$ semigroup of $G(1,\omega)$ class on $X$.
Such a semigroup is denoted by $\{S(t)\}_{t\ge0}$ and satisfies
\begin{align}\label{eqn:sg_estimate}
	\|S(t)\|_{X,X}\le e^{\omega t}~\mbox{with}~\omega:=\frac{1}{2}\left\|Bc\right\|_1,~~t\ge0.
\end{align}
\end{theorem}
Before proving Theorem \ref{thm:sg_generate}, the following lemma plays an important role:
\begin{lemma}\label{lem:C_inv}
The operator $C:a\in X\mapsto c*a$ has a dense range and is invertible under the following assumptions:
\[
	c\in\ell^1,\quad
%	~~~c_{-m}=\overline{c_m}~(m\in\mathbb{Z}),~~~\mbox{and}~~~
	|c_0|-\sum_{m\not=0}|c_m|>0.
\]
Moreover, the operator norm of $C^{-1}$ is estimated by
\[
	\left\|C^{-1}\right\|_{X,X}\le\frac{|c_0|^{-1}}{1-\frac{\sum_{m\not=0}|c_m|}{|c_0|}}.
\]
\end{lemma}
\begin{proof}
The operator $C$ is the linear bounded operator.
First, we consider the density of the range of $C$, say $R(C)$.
Assuming $R(C)\not=X$, there exists a nonzero $d\in X^*$ such that $\langle Ca,d\rangle_{X,X^*}=\sum_{k\in\mathbb{Z}}(c*a)_k\overline{d_k}=0$ for any $a\in X$.
%Furthermore, for arbitrary $\varepsilon>0$, we have an index $n$ such that $|d_n|\ge\|d\|_{X^\ast}-\varepsilon$ and $d_n\neq0$.
%Taking a sequence $a^*$ whose element is defined by
%\[
%	a^*_k=\left\{\begin{array}{ll}
%	0,&k\neq n,\\
%	(c*d)_n,&k=n,
%	\end{array}\right.
%\]
%we have $(c*d)_n=0$ from $\langle c*a^\ast,d\rangle_{X,X^\ast}=\langle a^\ast,c*d\rangle_{X,X^\ast}=0$.
%It follows
Taking $a=d$, it follows
\begin{align*}
	0&=\sum_{k\in\mathbb{Z}}\left(c*d\right)_k\overline{d_k}\\
	&=\sum_{k\in\mathbb{Z}}\left(c_0d_k+\sum_{m\neq k}c_{k-m}d_m\right)\overline{d_k}\\
	&\ge c_0\|d\|_{X^\ast}^2-\left(\sum_{m\not=0}|c_m|\right)\|d\|_{X^\ast}^2=\left(|c_0|-\sum_{m\not=0}|c_m|\right)\|d\|_{X^\ast}^2.
\end{align*}
%From the arbitrariness of $\varepsilon$, the last term becomes positive.
This implies $d=0$ if $|c_0|-\sum_{m\not=0}|c_m|>0$.
Then, it contradicts the assumption $d\neq 0$ and $R(C)=X$ holds.

Next, we consider the invertibility of the operator $C$.
For $a\in X$, we have
\[
	\left(a-c_0^{-1}c*a\right)_k=a_k-c_0^{-1}\sum_{m\in\mathbb{Z}}c_{k-m}a_m=-c_0^{-1}\sum_{m\neq k}c_{k-m}a_m.
\]
%Summing the absolute value of this with respect to $k$, we obtain
Taking the norm of this, Young's inequality yields
{\small\begin{align}\label{eqn:I-C_es}
	\left\|\left(I-c_0^{-1}C\right)a\right\|_X=\left(\sum_{k\in\mathbb{Z}}\left|\left(a-c_0^{-1}c*a\right)_k\right|^2\right)^{1/2}=\left(\sum_{k\in\mathbb{Z}}\left|c_0^{-1}\sum_{m\neq k}c_{k-m}a_m\right|^2\right)^{1/2}\le\frac{\sum_{m\not=0}|c_m|}{|c_0|}\|a\|_X<\|a\|_X.
\end{align}}
This implies $\left\|I-c_0^{-1}C\right\|_{X,X}<1$.
Then, the operator $c_0^{-1}C$ is invertible because its Neumann series converges in the operator norm on $X$.
Finally, from \eqref{eqn:I-C_es}, the operator norm of $C^{-1}$ is bounded by
\[
	\left\|C^{-1}\right\|_{X,X}\le\frac{|c_0|^{-1}}{1-\left\|I-c_0^{-1}C\right\|_{X,X}}\le\frac{|c_0|^{-1}}{1-\frac{\sum_{m\not=0}|c_m|}{|c_0|}}.
\]
This completes the proof.
\end{proof}
\begin{proof}[Proof of Theorem \ref{thm:sg_generate}]
First, from Lemma \ref{lem:dense} and \ref{lem:closedness}, the operator $A$ is the closed linear operator with dense domain $D(A)$ in $X$.
Next, we will prove that the operator $A$ is $\omega$-dissipative.
For $a\in D(A)$, let us define the element of the duality set $v\in F(a)$ as the sequence $a$ itself.
Thus, it sees that
%$v\in X^\ast$ as
%\begin{align}\label{eqn:v}
%	v_k:=\|a\|_X\mathop{\mathrm{sgn}}a_k=\left\{
%	\begin{array}{ll}
%	\frac{a_k}{|a_k|}\|a\|_X,&~a_k\not=0,\\
%	0,&~a_k=0.
%	\end{array}\right.
%\end{align}
%The norm of $v$ satisfies
%\[
%	\|v\|_{X^\ast}=\sup_{k\in\mathbb{Z}}|v_k|=\|a\|_X.
%\]
the dual product between $a$ and $v$ is
\begin{align*}
\langle a,v\rangle_{X,X^\ast}=\sum_{k\in\mathbb{Z}}a_k\overline{v_k}=\|a\|_X^2.
\end{align*}
%where $\overline{v_k}$ denotes the complex conjugate of $v_k$.
%Then, the sequence $v$ defined in \eqref{eqn:v} is an element of the duality set defined in \eqref{eqn:dualityset}, i.e., $v\in F(a)$.

\par
By using the sequence $v$, it follows for any $a\in D(A)$
\begin{align}\label{eqn:dproductAav}
\langle Aa,v\rangle_{X,X^\ast}&=-\sum_{k\in\mathbb{Z}}\left(\sum_{m\in\mathbb{Z}}c_{k-m}ima_m\right)\overline{v_k}\\\nonumber
&=-\sum_{k,m\in\mathbb{Z}}c_{k-m}\left(\frac{im}{2}-\frac{i}{2}(k-m)+\frac{ik}{2}\right)a_m\overline{v_k}\\\nonumber
&=-\frac{1}{2}\sum_{k,m\in\mathbb{Z}}\left(c_{k-m}(im)a_m\overline{v_k}-i(k-m)c_{k-m}a_m\overline{v_k}+ikc_{k-m}a_m\overline{v_k}\right)\\\nonumber
&=\frac{i}{2}\sum_{k,m\in\mathbb{Z}}(k-m)c_{k-m}a_m\overline{v_k}-\frac{i}{2}\sum_{k,m\in\mathbb{Z}}c_{k-m}\left(ma_m\overline{v_k}+a_mk\overline{v_k}\right)\\\nonumber
&=\frac{i}{2}\sum_{k,m\in\mathbb{Z}}(k-m)c_{k-m}a_m\overline{v_k}-\frac{i}{2}\sum_{k,m\in\mathbb{Z}}c_{k-m}\left(ma_m\overline{a_k}+a_mk\overline{a_k}\right).
\end{align}
Here, we set $d_{k,m}=ma_m\overline{a_k}+a_mk\overline{a_k}$.
By inverting the indexes of $d_{k,m}$, we observe that
\[
	d_{m,k}=ka_k\overline{a_m}+a_km\overline{a_m}=\overline{ma_m\overline{a_k}+a_mk\overline{a_k}}=\overline{d_{k,m}}.
\]
The second term in the last line of \eqref{eqn:dproductAav} follows
\begin{align}\label{eqn:impart}
\frac{i}{2}\sum_{k,m\in\mathbb{Z}}c_{k-m}d_{k,m}
&=\frac{i}{2}\sum_{k,m\in\mathbb{Z}}\left(\frac{1}{2}c_{k-m}d_{k,m}+\frac{1}{2}c_{m-k}d_{m,k}\right)\\\nonumber
&=\frac{i}{2}\sum_{k,m\in\mathbb{Z}}\left(\frac{1}{2}c_{k-m}d_{k,m}+\frac{1}{2}c_{-(k-m)}d_{m,k}\right)\\\nonumber
&=\frac{i}{2}\sum_{k,m\in\mathbb{Z}}\left(\frac{1}{2}c_{k-m}d_{k,m}+\frac{1}{2}\overline{c_{k-m}}\overline{d_{k,m}}\right)\in i\mathbb{R},
\end{align}
where $i\mathbb{R}$ denotes a set of purely imaginary numbers and we used the assumption\footnote{The variable coefficient $c(x)$ is the real-valued function.} $c_{-k}=\overline{c_k}$.
Then, we have the following estimate from \eqref{eqn:dproductAav} and \eqref{eqn:impart}:
\begin{align}\label{eqn:odissipative}
\mathrm{Re}\langle Aa,v\rangle_{X,X^\ast}&\le\frac{1}{2}\left|\sum_{k,m\in\mathbb{Z}}i(k-m)c_{k-m}a_m\overline{v_k}\right|\\\nonumber
&=\frac{1}{2}\left|\langle(Bc)*a,v\rangle_{X,X^\ast}\right|\\\nonumber
&\le\frac{1}{2}\left\|Bc\right\|_1\|a\|_X^2.
\end{align}
Taking $\omega=\frac{1}{2}\left\|Bc\right\|_1$, the operator $A$ is $\omega$-dissipative.

\par
Moreover, when we take $\lambda_0$ satisfying $\lambda_0>\omega$, from \eqref{eqn:odissipative} it follows for any $a\in D(A)$ and $v\in F(a)$
\begin{align*}%\label{eqn:coercive}
%\left|\langle a,(\lambda_0I-A)^*v\rangle_{X,X^\ast}\right|&=
\left|\langle(\lambda_0I-A)a,v\rangle_{X,X^\ast}\right|
&\ge\mathrm{Re}\left\langle(\lambda_0I-A)a,v\right\rangle_{X,X^\ast}\\\nonumber
&=\lambda_0\|a\|_{X}^2-\mathrm{Re}\langle Aa,v\rangle_{X,X^\ast}\\\nonumber
&\ge(\lambda_0-\omega)\|a\|_{X}^2.
\end{align*}
This implies that the null-space of $\lambda_0I-A$ has only the zero element, i.e., $N(\lambda_0I-A)=\{0\}$.
Then, $\lambda_0I-A$ is one-to-one and the range of $\lambda_0I-A$ is closed from the closed-range theorem (cf. \cite{bib:Brezis2010}).

\par
Finally, assuming that $R(\lambda_0I-A)\not=X$, there exists a nonzero element $b\in X^\ast$ such that
\begin{align}\label{eqn:la-Aa}
	\left\langle\lambda_0a-Aa,b\right\rangle_{X,X^\ast}=0,\quad \forall a\in D(A).
\end{align}
From the definition of $A$, it follows for $a\in D(A)$
\[
	Aa=-c*(Ba)=(Bc)*a-B\left(c*a\right).
\]
Then, \eqref{eqn:la-Aa} is equivalent to
\begin{align}\label{eqn:la-Aa_equiv}
	\left\langle\lambda_0a-(Bc)*a+B\left(c*a\right),b\right\rangle_{X,X^\ast}=0
	\iff\left\langle a,\lambda_0b-(Bc)*b\right\rangle_{X,X^\ast}=-\left\langle B\left(c*a\right),b\right\rangle_{X,X^\ast}.
\end{align}
By using Lemma \ref{lem:C_inv}, it follows from \eqref{eqn:la-Aa_equiv}
\begin{align}\label{eqn:B*b_estimate}
\left|\left\langle B\left(c*a\right),b\right\rangle_{X,X^\ast}\right|&=\left|\left\langle a,\lambda_0b-(Bc)*b\right\rangle_{X,X^\ast}\right|\\\nonumber
&=\left|\left\langle C^{-1}(c*a),\lambda_0b-(Bc)*b\right\rangle_{X,X^\ast}\right|\\\nonumber
&\le\left\|C^{-1}(c*a)\right\|_X\left\|\lambda_0b-(Bc)*b\right\|_{X^\ast}\\\nonumber
%&\le\frac{|c_0|^{-1}\left\|c*a\right\|_X}{1-\left\|I-c_0^{-1}C\right\|_{X,X}}\left\|\lambda_0b-(Bc)*b\right\|_{X^\ast}\\\nonumber
&\le\frac{|c_0|^{-1}\left\|c*a\right\|_X}{1-\frac{\sum_{m\not=0}|c_m|}{|c_0|}}\left\|\lambda_0b-(Bc)*b\right\|_{X^\ast}\\\nonumber
&\le\frac{|c_0|^{-1}}{1-\frac{\sum_{m\not=0}|c_m|}{|c_0|}}\left(\lambda_0+\|Bc\|_1\right)\|b\|_{X^\ast}\left\|c*a\right\|_X.
\end{align}
Since $\|Bc\|_1$ and $\|b\|_{X^\ast}$ are bounded, $b\in D(B^\ast)$ holds,
where $B^\ast$ denotes the adjoint of $B$.
We rewrite \eqref{eqn:la-Aa_equiv} as
\begin{align}\label{eqn:a_b_inp}
	\lambda_0\left\langle a,b\right\rangle_{X,X^\ast}&=\left\langle(Bc)*a,b\right\rangle_{X,X^\ast}-\left\langle B\left(c*a\right),b\right\rangle_{X,X^\ast}\\\nonumber
	&=\left\langle(Bc)*a,b\right\rangle_{X,X^\ast}-\left\langle c*a,B^\ast b\right\rangle_{X,X^\ast}.
\end{align}
%For arbitrary $\varepsilon>0$, we have an index $n$ such that $|b_n|\ge\|b\|_{X^\ast}-\varepsilon$.
%Taking a sequence $a^\ast$ whose element is defined by
%\[
%	a^*_k=\left\{\begin{array}{ll}
%	0,&k\neq n,\\
%	b_n,&k=n,
%	\end{array}\right.
%\]
%it is easy to see $a^\ast\in D(A)$.
Plugging $a=b$ in both sides of \eqref{eqn:a_b_inp}, the real part of \eqref{eqn:a_b_inp} yields
\begin{align}\label{eqn:|b_n|_es}
	\lambda_0\|b\|_{X^\ast}^2&=\mathrm{Re}\left(\left\langle(Bc)*b,b\right\rangle_{X,X^\ast}-\left\langle c*b,B^\ast b\right\rangle_{X,X^\ast}\right)\\\nonumber
	&=\mathrm{Re}\left(\sum_{k,m\in\mathbb{Z}}i(k-m)c_{k-m}b_m\overline{b_k}-\sum_{k,m\in\mathbb{Z}}c_{k-m}b_m\overline{(-ik)b_k}\right)\\\nonumber
	&\le\left|\sum_{k\neq m}\sum_{m\in\mathbb{Z}}i(k-m)c_{k-m}b_m\overline{b_k}-\sum_{k\neq m}\sum_{m\in\mathbb{Z}}c_{k-m}b_m\overline{(-ik)b_k}\right|\\\nonumber
	&\le\|b\|_{X^\ast}\left(\|Bc\|_1\|b\|_{X^\ast}+\sum_{m\not=0}|c_m|\left\|B^\ast b\right\|_{X^\ast}\right).
\end{align}
From \eqref{eqn:B*b_estimate}, we obtain an estimate of $\left\|B^\ast b\right\|_{X^\ast}$ 
\begin{align}\label{eqn:B*b_norm}
	\left\|B^\ast b\right\|_{X^\ast}=\sup_{0\neq x\in D(B)}\frac{\left\langle Bx,b\right\rangle_{X,X^\ast}}{\|x\|_X}\le\frac{|c_0|^{-1}}{1-\frac{\sum_{m\not=0}|c_m|}{|c_0|}}\left(\lambda_0+\|Bc\|_1\right)\|b\|_{X^\ast}.
\end{align}
Putting $\kappa:=\frac{\sum_{m\not=0}|c_m|}{|c_0|}$, it follows from \eqref{eqn:|b_n|_es} and \eqref{eqn:B*b_norm}
\begin{align*}
	\lambda_0\|b\|_{X^\ast}
	&\le\|Bc\|_1\|b\|_{X^\ast}+\frac{\kappa}{1-\kappa}\left(\lambda_0+\|Bc\|_1\right)\|b\|_{X^\ast}\\
	&=\left(\frac{\kappa}{1-\kappa}\lambda_0+\frac{1}{1-\kappa}\|Bc\|_1\right)\|b\|_{X^\ast}.
\end{align*}
This is equivalent to
\[
	\left\{(1-2\kappa)\lambda_0-\|Bc\|_1\right\}\|b\|_{X^\ast}\le0.
\]
The assumption \eqref{eqn:ass_c} yields $\|b\|_{X^\ast}=0$ if we take $\lambda_0$ satisfying 
\[
	\lambda_0>\frac{\|Bc\|_1}{1-2\kappa}(>\omega).
\]
This contradicts $b\neq0$ and $R(\lambda_0I-A)=X$ holds for such a $\lambda_0$.

\par
Therefore, from the Lumer-Phillips theorem (Theorem \ref{thm:LPtheorem}), the operator $A$ generates the $C_0$ semigroup of $G(1,\omega)$ class.
\end{proof}

\subsection{Bounded estimate of $C_0$ semigroup}
Since the behavior of the solution of \eqref{eqn:advec_Eq} follows characteristic curve, one may show that the solution becomes periodic in time.
In such a case, the estimate of the $C_0$ semigroup generated by $A$ is bounded (no exponential growth) by using the periodic property.

Let $T$ be a period such that $u(t+T)=u(t)$ for $t>0$.
If the solution of \eqref{eqn:advec_Eq} is a periodic solution of period $T$, $u(T,\cdot)=u_0$ holds, which implies $S(T)=I$.
Then, from the semigroup property (\textit{2.} in Definition \ref{def:semigroup}), we rewrite the estimate \eqref{eqn:sg_estimate} as
\begin{align}\label{eqn:sg_estimate_periodic}
	\|S(t)\|_{X,X}\le
%	\begin{cases}
%	e^{\omega t},~~0\le t<T,\\
	e^{\omega(t-nT)},~~nT\le t<(n+1)T,~n=0,1,2,\dots
%	\end{cases}
\end{align}
This estimate implies that the $C_0$ semigroup generated by $A$ becomes the $C_0$ semigroup of $G(M,0)$ class, where the constant $M$ is independent of $t$.
More precisely, we have $M=e^{\omega T}$.
The estimate \eqref{eqn:sg_estimate_periodic} is effective to enclose the solution of \eqref{eqn:advec_Eq} for long time because the estimate is independent of $t$.

\section{Rigorous error estimate using $C_0$ semigroup}\label{sec3}
In this section, we will derive a rigorous error estimate between the exact solution and a numerically computed approximate solution, which is our main result for verified numerical computations.
For a positive integer $N$, let
\[
	\tilde{u}(x,t)=\sum_{|k|\le N}\tilde{a}_k(t)e^{ikx}
\]
be a numerically computed approximate solution of \eqref{eqn:advec_Eq} and let $\tilde{a}(t)$ be the {\em infinite-dimensional extension} of $(\tilde{a}_k(t))_{|k|\le N}$, i.e., $\tilde{a}(t):=(\dots,0,0,\tilde{a}_{-N}(t),\dots,\tilde{a}_{N}(t),0,0\dots)$.
We also define a sequence of functions $z(t)=(z_k(t))_{k\in\mathbb{Z}}$ as $z(t)=a(t)-\tilde{a}(t)$.
%\[
%z_k(t)=
%\left\{
%\begin{array}{ll}
%a_k(t)-\tilde{a}_k(t),&|k|\le N,\\
%a_k(t),&|k|>N.
%\end{array}\right.
%\]
By using the sequence $z$, we rewrite \eqref{eqn:infODEs} as
\begin{align}\label{eqn:resODE}
\frac{d}{dt}z_k(t)+\left(c*(Bz(t))\right)_k=-\left\{\frac{d}{dt}\tilde{a}_k(t)+\left(c*(B\tilde{a}(t))\right)_k\right\},~k\in\mathbb{Z}.
\end{align}
%\begin{align}
%\left\{\begin{array}{ll}
%\displaystyle\frac{d}{dt}z_k(t)+\left(c*(Bz(t))\right)_k=-\left\{\frac{d}{dt}\tilde{a}_k(t)+\left(c*(B\tilde{a}(t))\right)_k\right\},&|k|\le N,\\[1mm]
%\displaystyle\frac{d}{dt}z_k(t)+\left(c*(Bz(t))\right)_k=0,&|k|>N.
%\end{array}\right.
%\end{align}
In the sequence space $X$, \eqref{eqn:resODE} is expressed by
\begin{align}\label{eqn:resODEinX}
\frac{d}{dt}z(t)-Az(t)=-\left(\frac{d}{dt}\tilde{a}(t)-A\tilde{a}(t)\right)~\mbox{in}~X,
\end{align}
where $A$ is defined in \eqref{eqn:opA}.
Since the operator $A$ generates the $C_0$ semigroup $\{S(t)\}_{t\ge0}$ of $G(1,\omega)$ class discussed in Section \ref{sec:generation_sg}, the solution of \eqref{eqn:resODEinX} is given by
\begin{align}\label{eqn:z(t)_form}
z(t)=S(t)z(0)+\int_{0}^{t}S(t-s)r(s)ds,
\end{align}
where $r(s)$ is the residual of the approximate solution defined by
\[
	r(s):=-\left(\frac{d}{ds}\tilde{a}(s)-A\tilde{a}(s)\right).
\]
From \eqref{eqn:sg_estimate} and \eqref{eqn:z(t)_form}, we have the following estimate for each $t>0$ 
\begin{align}\label{eqn:rigorous_err}
\left\|z(t)\right\|_X&\le\left\|S(t)z(0)\right\|_X+\int_{0}^{t}\left\|S(t-s)r(s)\right\|_Xds\\\nonumber
&\le e^{\omega t}\|z(0)\|_X+\int_{0}^{t}e^{\omega (t-s)}\|r(s)\|_Xds\\\nonumber
&\le e^{\omega t}\|z(0)\|_X+\left(\frac{e^{\omega t}-1}{\omega}\right)\|r\|_{C((0,t);X)},
\end{align}
where $\omega=\frac{1}{2}\left\|Bc\right\|_1$ and $\|r\|_{C((0,t);X)}:=\sup_{s\in(0,t)}\|r(s)\|_X$.
The estimate \eqref{eqn:rigorous_err} presents an error estimate between the Fourier coefficients of the exact solution of \eqref{eqn:advec_Eq} and the given approximate solution.
This is a core result of our method of verified computing.
In actual computations, we derive estimates of $\|z(0)\|_X$ and $\|r\|_{C((0,t);X)}$ rigorously, which is discussed in the next subsection.

\begin{remark}
If the solution of \eqref{eqn:advec_Eq} is a periodic solution of period $T$, the rigorous error estimate \eqref{eqn:rigorous_err} is improved by using \eqref{eqn:sg_estimate_periodic}
\begin{align*}%\label{eqn:rigorous_err_periodic}
\left\|z(t)\right\|_X\le
%\begin{cases}
%e^{\omega t}\|z(0)\|_X+\left(\frac{e^{\omega t}-1}{\omega}\right)\|r\|_{C((0,t);X)},~~~0\le t<T,\\
e^{\omega (t-nT)}\|z(0)\|_X+\left(\frac{e^{-\omega nT}}{\omega}\left(e^{\omega t}-1\right)\right)\|r\|_{C((0,t);X)}
%\end{cases}
\end{align*}
for $nT\le t<(n+1)T~(n=0,1,2,\dots)$.
\end{remark}

Finally, we have the following results, which show the existence locally in time of the solution of \eqref{eqn:advec_Eq} in the neighborhood of the numerically computed approximate solution.
The proof immediately follows from the above arguments.

\begin{theorem}\label{thm:main_thm}
For the advection equation \eqref{eqn:advec_Eq}, let $c(x)$ be a real-valued function and let the Fourier coefficients of $c(x)$ be $c=(c_k)_{k\in\mathbb{Z}}$.
If $\omega=\frac{1}{2}\left\|Bc\right\|_1$ is bounded and $|c_0|-2\sum_{m\not=0}|c_m|>0$ holds, then the time-dependent Fourier coefficients of the solution of \eqref{eqn:advec_Eq} exists in the function space $C((0,t_{\max});X)$ for $t_{\max}>0$.
The rigorous error estimate between the Fourier coefficients of the solution and its numerically computed approximate solution is given by
\begin{align*}
\left\|z\right\|_{C((0,t_{\max});X)}
&=\sup_{t\in(0,t_{\max})}\left\|z(t)\right\|_X\\
%&\le \sup_{t\in(0,t_{\max})}\left\{e^{\omega t}\|z(0)\|_X+\left(\frac{e^{\omega t}-1}{\omega}\right)\|r\|_{C((0,t_{\max});X)}\right\}\\
&\le e^{\omega t_{\max}}\|z(0)\|_X+\left(\frac{e^{\omega t_{\max}}-1}{\omega}\right)\|r\|_{C((0,t_{\max});X)}.
\end{align*}
\end{theorem}

\subsection{Rigorous estimation of initial error and residual}
\subsubsection{Approximate solution}
%By using Fourier-Chebyshev spectral method, we get the approximate solution.
The approximate solution is given by the Fourier series in space variable and the Chebyshev series in time variable.
We derive the following system of complex-valued ODEs via the discrete Fourier transform of \eqref{eqn:advec_Eq}:
\begin{align}\label{eqn:discreteODE}
	\frac{d}{dt}a_k(t)+\sum_{|m|\le N}c_{k-m}ima_m(t)=0,~|k|\le N.
\end{align}
The Fourier coefficients of the initial function and those of the variable coefficient are expressed by
\[
	a_k(0)=\frac{1}{2\pi}\int_{0}^{2\pi}u_0(x)e^{-ikx}dx\quad\mbox{and}\quad c_k=\frac{1}{2\pi}\int_{0}^{2\pi}c(x)e^{-ikx}dx,
\]
respectively.
We integrate \eqref{eqn:discreteODE} by using the integrator so-called {\em ode45 in Chebfun} \cite{bib:Driscoll2014}.
The numerically computed approximate solution is expressed by
\begin{align*}
\tilde{u}(x,t)&=\sum\limits_{|k|\le N}\tilde{a}_k(t)e^{ikx}=\sum\limits_{|k|\le N}\left(\sum_{l=0}^n \tilde{a}_{k,l}T_l(t) \right)e^{ikx},
\end{align*}
where $T_l(t)$ denotes the $l$-th Chebyshev polynomial of the first kind with respect to $t$ and
$\tilde{a}_{k,l}\in\mathbb{C}$.
Furthermore, the first derivative of the approximate solution with respect to the time variable is expressed for $\tilde{a}_k(t)=\sum_{l=0}^n \tilde{a}_{k,l}T_l(t)$ by
\[
	\frac{d}{dt}\tilde{a}_k(t)=\sum_{l=0}^{n-1} \tilde{a}_{k,l}^{(1)}T_l(t),
\]
where $\tilde{a}_{k,l}^{(1)}\in\mathbb{C}$ can be computed by an recursive algorithm (see, e.g., \cite[page 34]{bib:mason2002}).

\subsubsection{Initial error estimate}
To derive an upper bound of $\|z(0)\|_X$, we have
\begin{align}\label{eqn:initial_err}
\|z(0)\|_X&=\left(\sum_{k\in\mathbb{Z}}\left|a_k(0)-\tilde{a}_k(0)\right|^2\right)^{1/2}\\\nonumber
&\le\left(\sum_{|k|\le N}\left|a_k(0)-\tilde{a}_k(0)\right|^2\right)^{1/2}+\left(\sum_{|k|> N}|a_k(0)|^2\right)^{1/2}.
\end{align}
The first term represents the numerical (rounding) error, which is rigorously enclosed by {\em interval arithmetic}.
The second one represents the truncation error, which is expected to be small in practice.
For each given initial function, we estimate the last term of \eqref{eqn:initial_err} in the next section.

\subsubsection{Residual estimate}\label{sec:residual}
Let us consider the residual term.
For a fixed $s>0$, we have
\begin{align}\label{eqn:residual_es1}
&\left\|\frac{d}{ds}\tilde{a}(s)-A\tilde{a}(s)\right\|_X\\\nonumber
&=\left(\sum_{k\in\mathbb{Z}}\left|\frac{d}{ds}\tilde{a}_k(s)+\sum_{m\in\mathbb{Z}}c_{k-m}im\tilde{a}_m(s)\right|^2\right)^{1/2}\\\nonumber
&\le\left(\sum_{|k|\le N}\left|\frac{d}{ds}\tilde{a}_k(s)+\sum_{m\in\mathbb{Z}}c_{k-m}im\tilde{a}_m(s)\right|^2\right)^{1/2}+\left(\sum_{|k|> N}\left|\sum_{m\in\mathbb{Z}}c_{k-m}im\tilde{a}_m(s)\right|^2\right)^{1/2}\\\nonumber
&=\left(\sum_{|k|\le N}\left|\sum_{l=0}^{n-1}\tilde{a}_{k,l}^{(1)}T_l(s)+\sum_{m\in\mathbb{Z}}c_{k-m}im\left(\sum_{l=0}^n\tilde{a}_{m,l}T_l(s)\right)\right|^2\right)^{1/2}+\left(\sum_{|k|> N}\left|\sum_{m\in\mathbb{Z}}c_{k-m}im\tilde{a}_m(s)\right|^2\right)^{1/2}\\\nonumber
&=\left(\sum_{|k|\le N}\left|\sum_{l=0}^{n-1}\left(\tilde{a}_{k,l}^{(1)}+\sum_{|m|\le N}c_{k-m}im\tilde{a}_{m,l}\right)T_l(s)+\sum_{|m|\le N}c_{k-m}im\tilde{a}_{m,n}T_n(s)\right|^2\right)^{1/2}\\\nonumber
&\hphantom{=}\quad+\left(\sum_{|k|> N}\left|\sum_{m\in\mathbb{Z}}c_{k-m}im\tilde{a}_m(s)\right|^2\right)^{1/2}\\\nonumber
&=\left(\sum_{|k|\le N}\left|\sum_{l=0}^{n}d_{k,l}T_l(s)\right|^2\right)^{1/2}+\left(\sum_{|k|> N}\left|\sum_{m\in\mathbb{Z}}c_{k-m}im\tilde{a}_m(s)\right|^2\right)^{1/2},\\\nonumber
%\le&\sum_{|k|\le N}\sum_{l=0}^{n}\left|d_{k,l}\right|+\sum_{|k|> N}\left|\sum_{m\in\mathbb{Z}}c_{k-m}im\tilde{a}_m(s)\right|
\end{align}
where we denote
\[
	d_{k,l}:=\left\{\begin{array}{ll}
	\displaystyle\tilde{a}_{k,l}^{(1)}+\sum_{|m|\le N}c_{k-m}im\tilde{a}_{m,l}&(0\le l<n),\\[2mm]
	\displaystyle\sum_{|m|\le N}c_{k-m}im\tilde{a}_{m,n}&(l=n).
	\end{array}\right.
\]
The second term of the last line in \eqref{eqn:residual_es1} is estimated as follows:
let us divide the sequence $c=c^{(N)}+c^{(\infty)}$, where $c^{(N)}:=(\dots,0,c_{-N},\dots,c_{N},0,\dots)$ and $c^{(\infty)}:=(\dots,c_{-N-1},0,\dots,0,c_{N+1},\dots)$.
We have
{\small\begin{align*}
\left(\sum_{|k|> N}\left|\sum_{m\in\mathbb{Z}}c_{k-m}im\tilde{a}_m(s)\right|^2\right)^{1/2}
&=\left(\sum_{|k|> N}\left|\left(c*\left(B\tilde{a}(s)\right)\right)_k\right|^2\right)^{1/2}\\
&=\left(\sum_{|k|> N}\left|\left(\left(c^{(N)}+c^{(\infty)}\right)*\left(B\tilde{a}(s)\right)\right)_k\right|^2\right)^{1/2}\\
&\le\left(\sum_{|k|> N}\left|\left(c^{(N)}*\left(B\tilde{a}(s)\right)\right)_k\right|^2\right)^{1/2}+\left(\sum_{|k|> N}\left|\left(c^{(\infty)}*\left(B\tilde{a}(s)\right)\right)_k\right|^2\right)^{1/2}\\
&\le\left(\sum_{N<|k|\le 2N}\left|\sum_{\substack{k_1+k_2=k,\\|k_1|,|k_2|\le N}}c_{k_1}ik_2\tilde{a}_{k_2}(s)\right|^2\right)^{1/2}+\left\|c^{(\infty)}\right\|_1\left\|B\tilde{a}(s)\right\|_X.
\end{align*}}%
The last term is rigorously computable by using interval arithmetic and the truncated error estimate of the sequence $c$.
Finally, for some $t_{\max}>0$, we obtain an upper bound of the residual term
\begin{align}\label{eqn:residual_es}
	&\|r\|_{C((0,t_{\max});X)}\\\nonumber
	&=\sup_{s\in(0,t_{\max})}\left\|\frac{d}{ds}\tilde{a}(s)-A\tilde{a}(s)\right\|_X\\\nonumber
	&\le\sup_{s\in(0,t_{\max})}\left[\left(\sum_{|k|\le N}\left|\sum_{l=0}^{n}d_{k,l}T_l(s)\right|^2\right)^{1/2}+\left(\sum_{N<|k|\le 2N}\left|\sum_{\substack{k_1+k_2=k,\\|k_1|,|k_2|\le N}}c_{k_1}ik_2\sum_{l=0}^n\tilde{a}_{k_2,l}T_l(s)\right|^2\right)^{1/2}\right.\\\nonumber
	&\quad+\left.\left\|c^{(\infty)}\right\|_1\left(\sum_{|k|\le N}\left|k\sum_{l=0}^n\tilde{a}_{k,l}T_l(s)\right|^2\right)^{1/2}\right]\\\nonumber
	&\le\left(\sum_{|k|\le N}\left(\sum_{l=0}^{n}\left|d_{k,l}\right|\right)^2\right)^{1/2}+\left(\sum_{N<|k|\le 2N}\left(\sum_{l=0}^n\left|\gamma_{k,l}\right|\right)^2\right)^{1/2}+\left\|c^{(\infty)}\right\|_X\left(\sum_{|k|\le N}\left(\left|k\right|\sum_{l=0}^n\left|\tilde{a}_{k,l}\right|\right)^2\right)^{1/2},
\end{align}
where $\gamma_{k,l}$ is denoted by
\[
	\gamma_{k,l}=\sum_{\substack{k_1+k_2=k,\\|k_1|,|k_2|\le N}}c_{k_1}ik_2\tilde{a}_{k_2,l}.
\]
\section{Numerical results}\label{sec:num_result}
In this section, we show three examples to demonstrate the efficiency of our verified numerical computations.
All computations are carried out on Windows 10, Intel(R) Core(TM) i7-6700K CPU @ 4.00GHz, and MATLAB 2017a with INTLAB - INTerval LABoratory \cite{bib:intlab} version 10.2 and Chebfun - numerical computing with functions \cite{bib:Driscoll2014} version 5.6.0.
All codes used to produce the results in this section are freely available from \cite{bib:codes}.

\subsection{Example 1}
As the first example, we consider the initial-boundary value problem \eqref{eqn:advec_Eq} with
%\begin{align*}
%c(x)=\frac{1}{5}+\sin^2(x-1),~~u_0(x)\approx e^{-100(x-1)^2}.
%\end{align*}
\[
	c(x)=0.51+\sin^2(x-1),\quad u_0(x)=\sum_{k\in\mathbb{Z}}a_k(0)e^{ikx},\quad a_k(0)=\frac{1}{20\sqrt{\pi}}e^{-\frac{k^2}{400}-ik}.
\]
This initial function is an analytic function satisfying $u_0(x)\approx e^{-100(x-1)^2}$.
Because the Fourier coefficients of $e^{-100(x-1)^2}$ are
{
\[
	\frac{1}{2\pi}\int_{0}^{2\pi}e^{-100(x-1)^2}e^{-ikx}dx=\frac{e^{-\frac{k^2}{400}-ik}}{40\sqrt{\pi}}\left(\operatorname{erf}\left(\frac{ik+400\pi-200}{20}\right)-\operatorname{erf}\left(\frac{ik-200}{20}\right)\right)~(k\in\mathbb{Z}),
\]}%
%where $\operatorname{erf}(x)=\int_{0}^xe^{-s^2}ds$.
we approximate the term $\operatorname{erf}\left(\frac{ik+400\pi-200}{20}\right)-\operatorname{erf}\left(\frac{ik-200}{20}\right)$ as $2$, where $\operatorname{erf}(x)=\frac{2}{\sqrt{\pi}}\int_{0}^xe^{-s^2}ds$.
In \cite{bib:Trefethen2000}, the function $e^{-100(x-1)^2}$ is said that \textit{this function is not mathematically periodic, but it is so close to zero at the ends of the interval that it can be regarded as periodic in practice}.
Then, we set such an initial function that is rigorously periodic and is close to $e^{-100(x-1)^2}$.
The truncated error of the initial function in \eqref{eqn:initial_err} is estimated by
\begin{align}\label{eqn:truncate_ex1}
	\left(\sum_{|k|> N}|a_k(0)|^2\right)^{1/2}&=\frac{1}{20\sqrt{\pi}}\left(\sum_{|k|> N}\left|e^{-\frac{k^2}{400}}\right|^2\right)^{1/2}\\\nonumber
	&=\frac{1}{10\sqrt{2\pi}}\left(\sum_{k=N+1}^\infty\left|e^{-\frac{k^2}{400}}\right|^2\right)^{1/2}\\\nonumber
	&\le\frac{1}{10\sqrt{2\pi}}\left(\int_{N}^{\infty}e^{-\frac{x^2}{200}}dx\right)^{1/2}\\\nonumber
%	&=\frac{1}{10\sqrt{2\pi}}\left(\int_{0}^{\infty}e^{-\frac{x^2}{200}}dx-\int_{0}^{N}e^{-\frac{x^2}{200}}dx\right)^{1/2}\\\nonumber
	&=\frac{1}{2}\left(\operatorname{erfc}\left(\frac{N}{10\sqrt{2}}\right)\right)^{1/2},
\end{align}
where $\operatorname{erfc}(x)=1-\operatorname{erf}(x)$.
Furthermore, the Fourier coefficients of $c(x)$ are given by
\[
	c_k=\begin{cases}
	-\frac{e^{2i}}{4},&k=-2,\\
	1.01,&k=0,\\
	-\frac{e^{-2i}}{4},&k=2,\\
	0,&\mbox{otherwise}.
	\end{cases}
\]
It is shown that the sequence $c=(c_k)_{k\in\mathbb{Z}}$ satisfies $Bc\in\ell^1$ and \eqref{eqn:ass_c}.
Theorem \ref{thm:sg_generate} follows that the operator $A$ defined in \eqref{eqn:opA} generates the $C_0$ semigroup of $G(1,\omega)$ class on the sequence space $X$ with
%We obtain
\[
	\omega=\frac{1}{2}\|Bc\|_1=\frac{1}{2}.
\]
The profile of the variable coefficient $c(x)$ and that of the initial function $u_0(x)$ are displayed in Figure \ref{fig:fig1}.
\begin{figure}[htbp]\em
\centering
\includegraphics[width=0.8\textwidth]{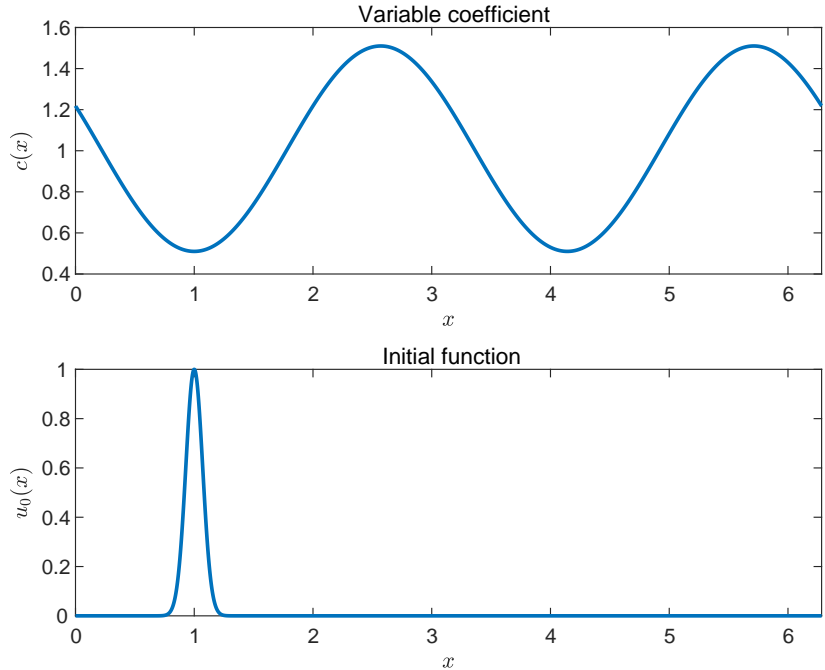}
\caption{Profile of the variable coefficient (upper) and that of the initial function (lower) in Example 1.}
\label{fig:fig1}
\end{figure}
In Figure \ref{fig:fig2}, we plot the behavior of the numerically computed approximate solution $\tilde{u}$.
\begin{figure}[htbp]\em
\centering
\includegraphics[width=0.8\textwidth]{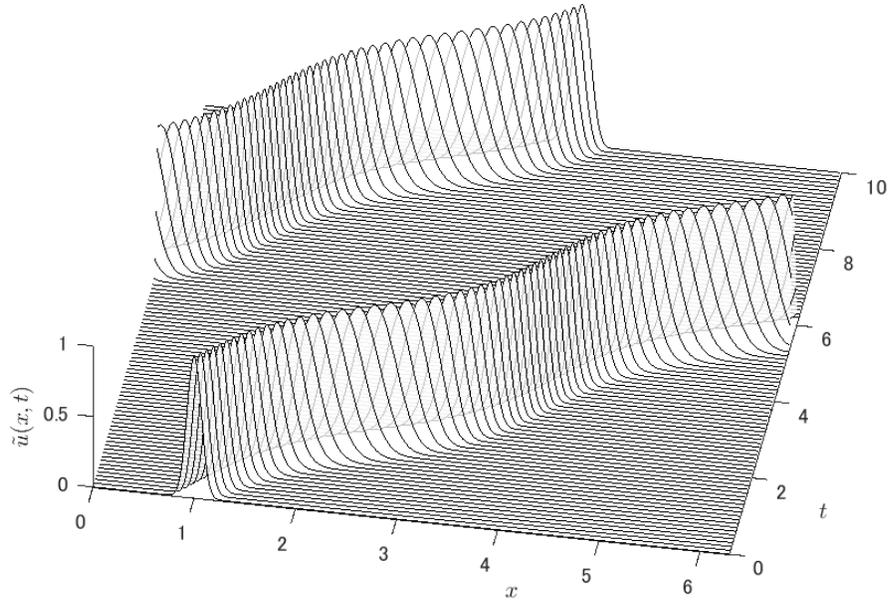}
\caption{Numerically computed approximate solution of \eqref{eqn:advec_Eq} in Example 1.}
\label{fig:fig2}
\end{figure}

Figure \ref{fig:fig3} shows rigorous upper bounds of $\|z(0)\|_X$ using \eqref{eqn:initial_err} with \eqref{eqn:truncate_ex1}.
A remarkable accuracy can be seen in Figure \ref{fig:fig3}, which illustrates the efficiency of the spectral method.
\begin{figure}[htbp]\em
{\centering
\includegraphics[width=0.75\textwidth]{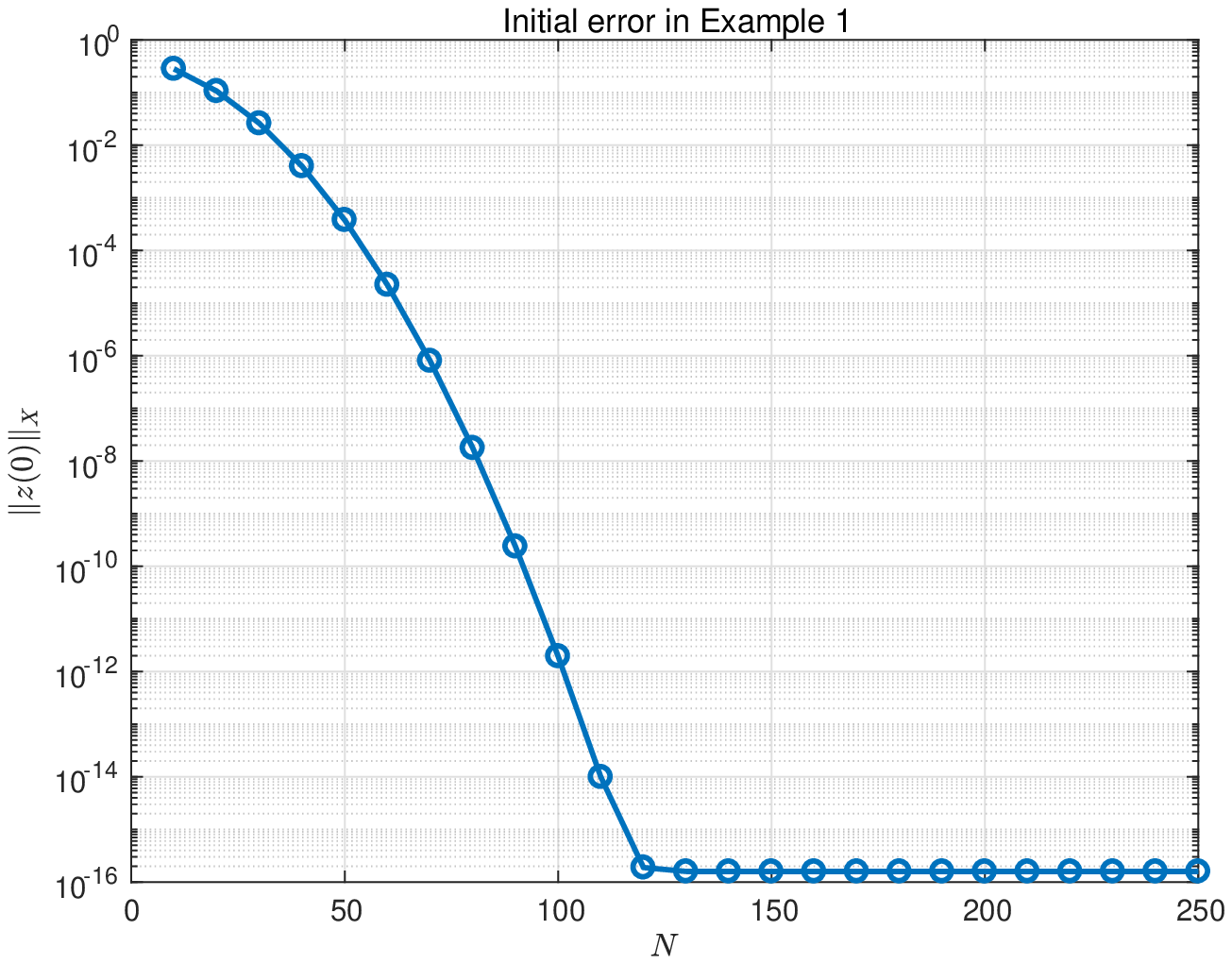}
\caption{Error estimates of the initial function $\|z(0)\|_X$ in Example 1.}
\label{fig:fig3}}
Rigorous upper bounds of \eqref{eqn:initial_err} with \eqref{eqn:truncate_ex1} are shown when $N=10,20,\dots,250$.
The accuracy of the error estimate is sufficiently small and the rounding errors take over when $N\ge120$.
Furthermore, the decay rate of the estimate looks exponential order, which is called ``spectral accuracy'' in \cite{bib:Trefethen2000}.
\end{figure}
Figure \ref{fig:fig4} also shows rigorous numerical results of the residual estimate discussed in Section \ref{sec:residual}.
This result implies that the number of Chebyshev basses in order to sufficiently reduce the residual estimate is increasing as $t_{\max}$ getting larger.
\begin{figure}[htbp]\em
{\centering
\includegraphics[width=0.8\textwidth]{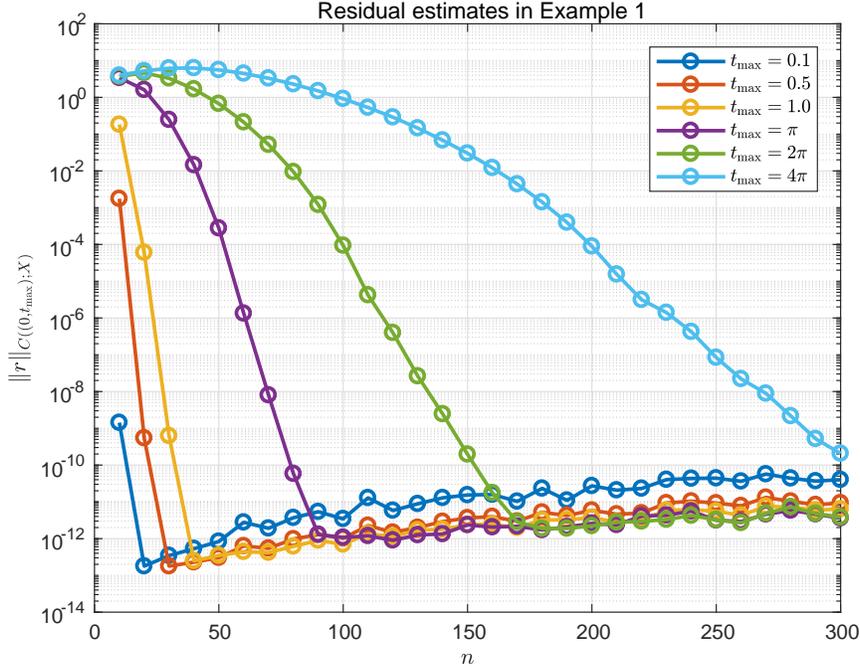}
\caption{Residual estimates $\|r\|_{C((0,t_{\max});X)}$ ($N=120$) in Example 1.}
\label{fig:fig4}}
Rigorous upper bounds of the residual estimate \eqref{eqn:residual_es} are shown when $n=10,20,\dots,300$ varying $t_{\max}=0.1,0.5,1,\pi,2\pi,4\pi$.
The accuracy becomes small and the rounding errors take over for large $n$.
The decay rate also looks exponential order with respect to $n$.
In particular, when one take the sufficiently large number of Chebyshev basses, the accuracy of the residual estimate becomes tiny even if $t_{\max}$ is large.
\end{figure}

Consequently, in Table \ref{Tab:Ex1}, we sum up each numerical result of our verified computing when $t_{\max}=0.1$, $0.5$, $\pi$, $2\pi$, $4\pi$ and $N=120$.
\begin{table}[ht]\em
\caption{Numerical results of verified computing  in Example 1.}
\centering
%    \resizebox{\textwidth}{!}{
    \begin{tabular}{rrrrrrrrr}
    \hline
    \multicolumn{1}{c}{$t_{\max}$} & \multicolumn{1}{c}{$N$} & \multicolumn{1}{c}{$n$} & \multicolumn{1}{c}{initial error} & \multicolumn{1}{c}{residual} & \multicolumn{1}{c}{error}  & \multicolumn{1}{c}{app.~time} & \multicolumn{1}{c}{exec.~time}& \multicolumn{1}{c}{ratio}\\
    \hline
%    0.1 & 120 & 15 & 1.2212e-15 & 5.7518e-13 & 6.0264e-14 & 0.4059 & 0.5009 & 1.23\\
%    0.5 & 120 & 28 & 1.2212e-15 & 6.6848e-13 & 3.813e-13 & 1.6508 & 1.8086 & 1.10\\
%    1.0 & 120 & 39 & 1.2212e-15 & 1.1576e-12 & 1.504e-12 & 3.2064 & 3.4031 & 1.06\\
%    $\pi$ & 120 & 97 & 1.2212e-15 & 4.6524e-12 & 3.5462e-11 & 15.3547 & 15.7950 & 1.03\\
%    $2\pi$ & 120 & 182 & 1.2212e-15 & 1.159e-11 & 5.1324e-10 & 28.0592 & 28.7940 & 1.03\\
%    $4\pi$ & 120 & 355 & 1.2212e-15 & 4.0631e-11 & 4.3435e-08 & 60.0733 & 61.5937 & 1.03\\
%0.1 & 120 & 15 & 1.2212e-15 & 5.7518e-13 & 6.0264e-14 & 0.2291 & 0.2638 & 1.151\\
%0.5 & 120 & 28 & 1.2212e-15 & 6.6848e-13 & 3.813e-13 & 1.1916 & 1.2306 & 1.033\\
%1.0 & 120 & 39 & 1.2212e-15 & 1.1576e-12 & 1.504e-12 & 2.3786 & 2.4154 & 1.015\\
%$\pi$ & 120 & 97 & 1.2212e-15 & 4.6524e-12 & 3.5462e-11 & 7.6551 & 7.6984 & 1.006\\
%$2\pi$ & 120 & 182 & 1.2212e-15 & 1.159e-11 & 5.1324e-10 & 15.2648 & 15.3220 & 1.004\\
%$4\pi$ & 120 & 355 & 1.2212e-15 & 4.0631e-11 & 4.3435e-08 & 30.3761 & 30.4519 & 1.002\\
0.1 & 120 & 15 & 1.888e-16 & 6.6975e-14 & 7.0663e-15 & 0.2258 & 0.2634 & 1.166\\
0.5 & 120 & 28 & 1.888e-16 & 1.331e-13 & 7.5849e-14 & 1.0445 & 1.0881 & 1.042\\
1.0 & 120 & 39 & 1.888e-16 & 2.0349e-13 & 2.6433e-13 & 2.0794 & 2.1221 & 1.021\\
$\pi$ & 120 & 97 & 1.888e-16 & 7.3368e-13 & 5.5922e-12 & 6.8228 & 6.8702 & 1.007\\
$2\pi$ & 120 & 182 & 1.888e-16 & 1.6846e-12 & 7.46e-11 & 13.6139 & 13.6696 & 1.004\\
$4\pi$ & 120 & 355 & 1.888e-16 & 6.5085e-12 & 6.9576e-09 & 27.2039 & 27.2845 & 1.003\\
    \hline
    \end{tabular}%
%    }
\label{Tab:Ex1}
\end{table}
Here, \emph{``initial error''}, \emph{``residual''}, and \emph{``error''} denote the upper bound of $\|z(0)\|_X$, that of the residual $\|r\|_{C((0,t_{\max});X)}$, and that of $\|z\|_{C((0,t_{\max});X)}$ given in Theorem \ref{thm:main_thm}, respectively.
The \emph{``app.~time''} and \emph{``exec.~time''} are results of computational time on the second time scale for getting an approximate solution by {\em ode45 in Chebfun} and those of all computational time for our verified computing, respectively.
The ``ratio'' presents an additional ratio of verified computing given by the value of {\em ``exec.~time'' divided by ``app.~time''}.
The main feature of these results is the accuracy of the residual estimate, which is remarkably high accurate for verified computing to PDEs.
It is worth noting that the rigorous error estimate is still small even if $t_{\max}$ is getting larger.
Another feature of the results is the speed of verified numerical computations.
In Table \ref{Tab:Ex1}, we observe that it requires at most tens of percent of additional time with respect to the execute time of the ODE integrator in \emph{Chebfun}.
This indicates that our method of verified computing is executed as fast as (non-rigorous) numerical computations.

\subsection{Example 2}
The second example is the initial-boundary value problem \eqref{eqn:advec_Eq} with
\begin{align*}
c(x)=1+0.49\cos 2x,~~u_0(x)=\frac{3}{5+4\cos x}.
\end{align*}
The profile of the variable coefficient $c(x)$ and that of the initial function $u_0(x)$ are displayed in Figure \ref{fig:fig5}.
\begin{figure}[htbp]\em
\centering
\includegraphics[width=0.8\textwidth]{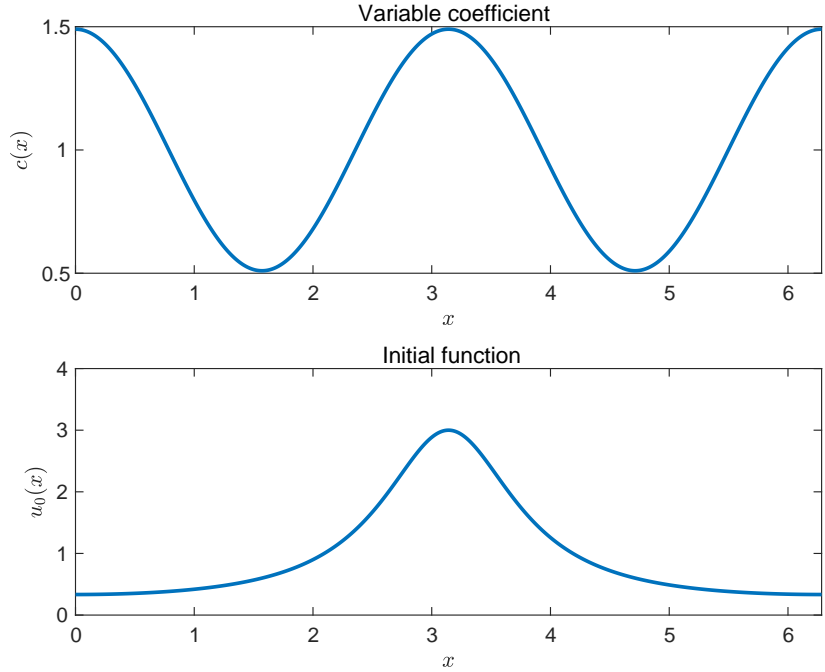}
\caption{Profile of the variable coefficient (upper) and that of the initial function (lower) in Example 2.}
\label{fig:fig5}
\end{figure}
%
%This initial function is based on the Poisson kernel defined by
%\[
%	f_{\sigma}(x)=\frac{1-\sigma^2}{1-2\sigma\cos x+\sigma^2}=\sum_{k\in\mathbb{Z}}\sigma^{|k|}e^{ikx},~0\le\sigma<1.
%\]
%Setting $\sigma=1/2$,
The Fourier coefficients of the initial function are given by $a_k(0)=(-2)^{-|k|}$.
Because the floating-point number is binary format with base $2$, the rounding error in calculating the coefficients becomes zero.
Then, the initial error is only the truncated error.
Such a truncated error of the initial function in \eqref{eqn:initial_err} is estimated by
\begin{align}\label{eqn:truncate_ex2}
\left(\sum_{|k|> N}|a_k(0)|^2\right)^{1/2}=\left(2\sum_{k=N+1}^\infty\left|(-2)^{-k}\right|^2\right)^{1/2}=\sqrt{\frac{2}{3}}\,2^{-N}.
\end{align}
Furthermore, the Fourier coefficients of $c(x)$ are given by
\[
c_k=\begin{cases}
\frac{0.49}{2},&k=-2,\\
1,&k=0,\\
\frac{0.49}{2},&k=2,\\
0,&\mbox{otherwise}.
\end{cases}
\]
The sequence $c=(c_k)_{k\in\mathbb{Z}}$ immediately satisfies $Bc\in \ell^1$ and \eqref{eqn:ass_c}.
Theorem \ref{thm:sg_generate} follows that the operator $A$ defined in \eqref{eqn:opA} generates the $C_0$ semigroup of $G(1,\omega)$ class on the sequence space $X$ with
%We obtain
\[
\omega=\frac{1}{2}\|Bc\|_1=0.49.
\]
In Figure \ref{fig:fig6}, we plot the behavior of the numerically computed approximate solution $\tilde{u}$.
\begin{figure}[htbp]\em
\centering
\includegraphics[width=0.8\textwidth]{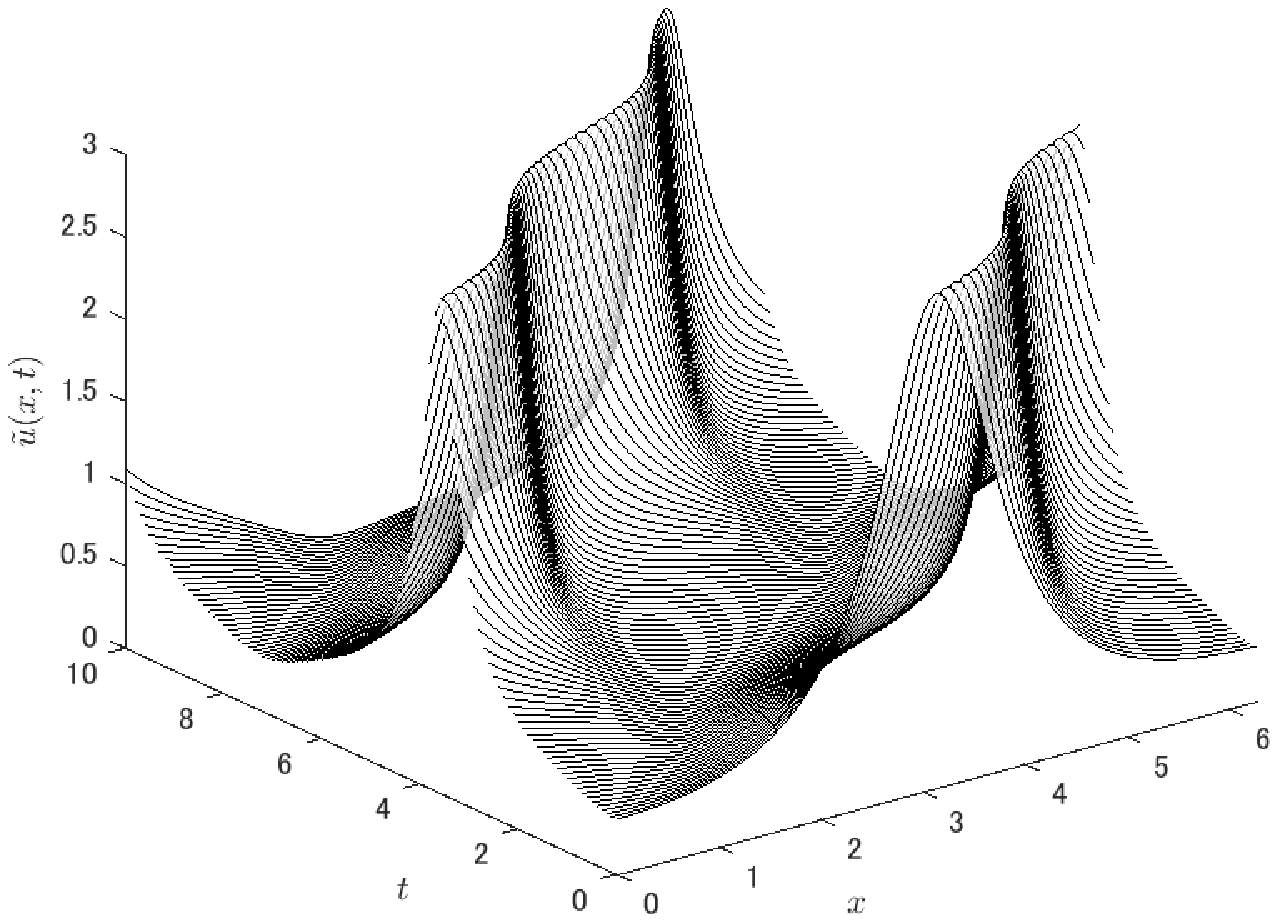}
\caption{Numerically computed approximate solution of \eqref{eqn:advec_Eq} in Example 2.}
\label{fig:fig6}
\end{figure}

Figure \ref{fig:fig7} shows rigorous upper bounds of the residual estimate discussed in Section \ref{sec:residual}.
The results are almost same as that in the previous example.
The residual estimate becomes sufficiently small as increasing the number of Chebyshev basses even if $t_{\max}=4\pi(\approx 12.57)$.
\begin{figure}[htbp]\em
{\centering
\includegraphics[width=0.8\textwidth]{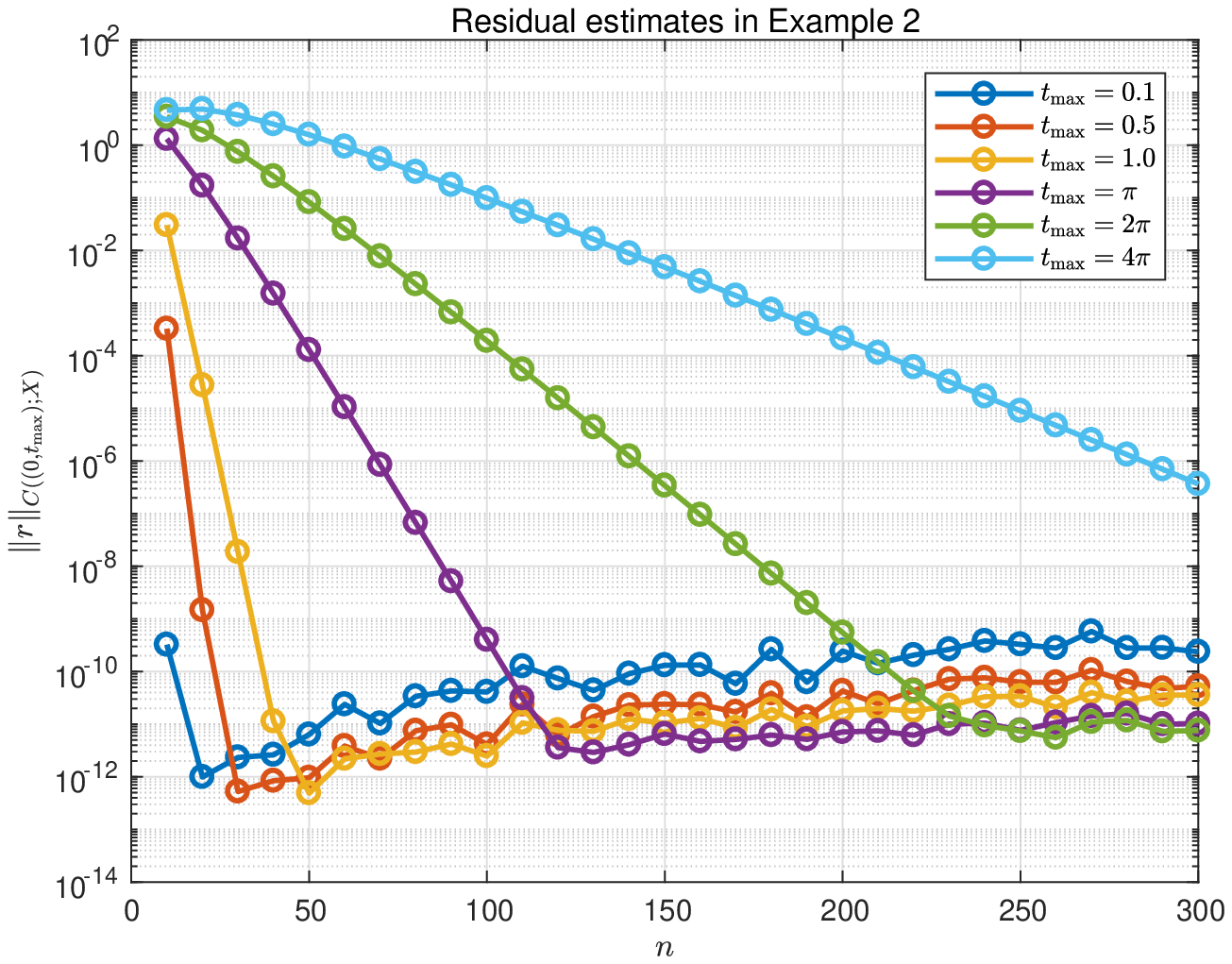}
\caption{Residual estimates $\|r\|_{C((0,t_{\max});X)}$ ($N=150$) in Example 2.}
\label{fig:fig7}}
Rigorous upper bounds of the residual estimate \eqref{eqn:residual_es} are shown when $n=10,20,\dots,300$ varying $t_{\max}=0.1,0.5,1,\pi,2\pi,4\pi$.
The decay rate also looks exponential order with respect to $n$.
%For the sufficiently large number of Chebyshev basses, the accuracy of the residual estimate  becomes tiny even if $t_{\max}$ is large.
\end{figure}
Table \ref{Tab:Ex2} lists the \emph{best} numerical result of verified computing when $t_{\max}=0.1,~0.5,~\pi,~2\pi,~4\pi$.
\begin{table}[ht]\em
\caption{Numerical results of verified computing in Example 2.}
\centering
%    \resizebox{\textwidth}{!}{
    \begin{tabular}{rrrrrrrrr}
    \hline
    \multicolumn{1}{c}{$t_{\max}$} & \multicolumn{1}{c}{$N$} & \multicolumn{1}{c}{$n$} & \multicolumn{1}{c}{initial error} & \multicolumn{1}{c}{residual} & \multicolumn{1}{c}{error}  & \multicolumn{1}{c}{app.~time} & \multicolumn{1}{c}{exec.~time}& \multicolumn{1}{c}{ratio}\\
    \hline
%0.1 & 70 & 14 & 1.6941e-21 & 8.8965e-13 & 9.1181e-14 & 0.6037 & 0.6326 & 1.048\\
%0.5 & 100 & 29 & 1.5777e-30 & 1.3559e-12 & 7.6819e-13 & 8.5213 & 8.5660 & 1.005\\
%1.0 & 150 & 50 & 1.4013e-45 & 1.7674e-12 & 2.2808e-12 & 20.9527 & 21.0310 & 1.004\\
%$\pi$ & 176 & 130 & 2.0881e-53 & 9.3794e-12 & 7.0091e-11 & 59.7052 & 59.8567 & 1.003\\
%$2\pi$ & 204 & 265 & 7.7788e-62 & 2.0735e-11 & 8.7728e-10 & 162.8326 & 162.9968 & 1.001\\
%$4\pi$ & 200 & 539 & 1.2446e-60 & 6.7953e-11 & 6.5354e-08 & 390.1671 & 390.6685 & 1.001\\
%0.1 & 70 & 14 & 1.6941e-21 & 8.8965e-13 & 9.1181e-14 & 0.4870 & 0.5052 & 1.037\\
%0.5 & 100 & 29 & 1.5777e-30 & 1.3559e-12 & 7.6819e-13 & 2.2773 & 2.3061 & 1.013\\
%1.0 & 150 & 50 & 1.4013e-45 & 1.7674e-12 & 2.2808e-12 & 6.6088 & 6.6479 & 1.006\\
%$\pi$ & 176 & 130 & 2.0881e-53 & 9.3794e-12 & 7.0091e-11 & 23.1366 & 23.1950 & 1.003\\
%$2\pi$ & 204 & 265 & 7.7788e-62 & 2.0735e-11 & 8.7728e-10 & 84.6759 & 84.7677 & 1.001\\
%$4\pi$ & 200 & 539 & 1.2446e-60 & 6.7953e-11 & 6.5354e-08 & 162.9074 & 163.0771 & 1.001\\
0.1 & 110 & 13 & 2.2662e-17 & 3.1829e-13 & 3.2645e-14 & 0.5219 & 0.5533 & 1.060\\
0.5 & 152 & 29 & 1.0806e-23 & 3.8705e-13 & 2.1929e-13 & 3.4671 & 3.5052 & 1.011\\
1.0 & 194 & 46 & 5.1528e-30 & 6.4904e-13 & 8.3755e-13 & 12.7719 & 12.8210 & 1.004\\
$\pi$ & 223 & 121 & 2.2239e-34 & 2.7793e-12 & 2.077e-11 & 48.1160 & 48.1829 & 1.001\\
$2\pi$ & 190 & 247 & 2.0611e-29 & 8.2819e-12 & 3.504e-10 & 74.3893 & 74.4721 & 1.001\\
$4\pi$ & 200 & 536 & 6.441e-31 & 1.6484e-11 & 1.5854e-08 & 166.2447 & 166.4098 & 1.001\\
    \hline
    \end{tabular}%
%    }
\label{Tab:Ex2}
\end{table}
Here, ``\emph{best}'' means that our code returns the smallest rigorous error estimate given in Theorem \ref{thm:main_thm}.
As indicated in Example 1, the high accuracy and high speed of verified numerical computations catch one's attention to illustrate the efficiency of the method.

\subsection{Example 3}
For the final example, we consider a slightly more complicated initial-boundary value problem \eqref{eqn:advec_Eq} than that in the previous two examples, where
\begin{align*}
c(x)=-1+0.3\sin 3x-0.19\cos 2x,~~u_0(x)=\frac{3}{5+4\cos x}.
\end{align*}
The profile of the variable coefficient $c(x)$ and that of the initial function $u_0(x)$ are displayed in Figure \ref{fig:fig8}.
\begin{figure}[htbp]\em
\centering
\includegraphics[width=0.8\textwidth]{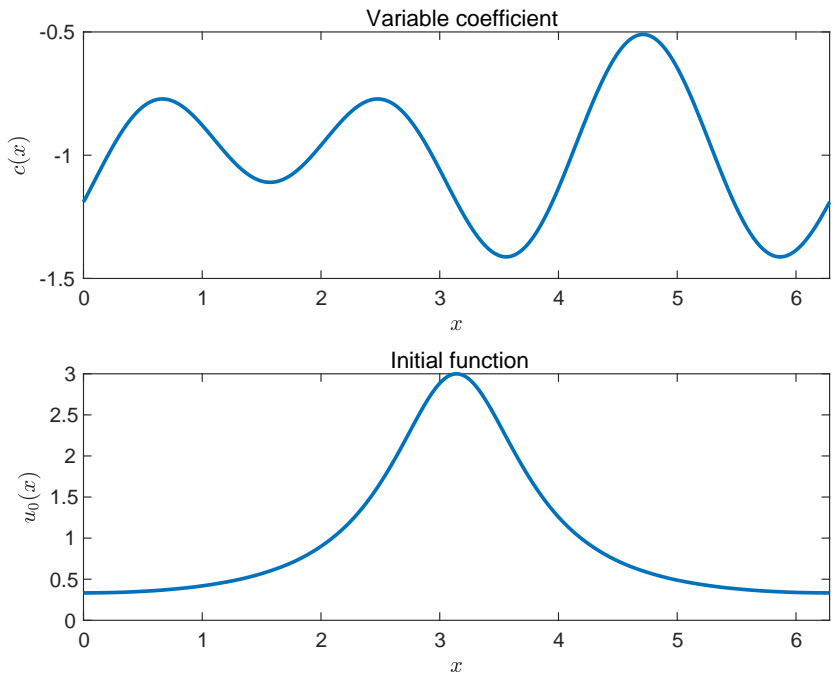}
\caption{Profile of the variable coefficient (upper) and that of the initial function (lower) in Example 3.}
\label{fig:fig8}
\end{figure}
The truncated error of the initial function in \eqref{eqn:initial_err} is the same as that in \eqref{eqn:truncate_ex2}.
The Fourier coefficients of $c(x)$ are given by
\[
c_k=\begin{cases}
\frac{0.3}{2i},&k=-3,\\
-\frac{0.19}{2},&k=-2,\\
-1,&k=0,\\
-\frac{0.19}{2},&k=2,\\
-\frac{0.3}{2i},&k=3,\\
0,&\mbox{otherwise}.
\end{cases}
\]
The sequence $c=(c_k)_{k\in\mathbb{Z}}$ immediately satisfies $Bc\in \ell^1$ and \eqref{eqn:ass_c}.
Theorem \ref{thm:sg_generate} follows that the operator $A$ defined in \eqref{eqn:opA} generates the $C_0$ semigroup of $G(1,\omega)$ class on the sequence space $X$ with
%We obtain
\[
	\omega=\frac{1}{2}\|Bc\|_1=0.64.
\]
In Figure \ref{fig:fig9}, we plot the behavior of the numerically computed approximate solution $\tilde{u}$.
\begin{figure}[htbp]\em
\centering
\includegraphics[width=0.8\textwidth]{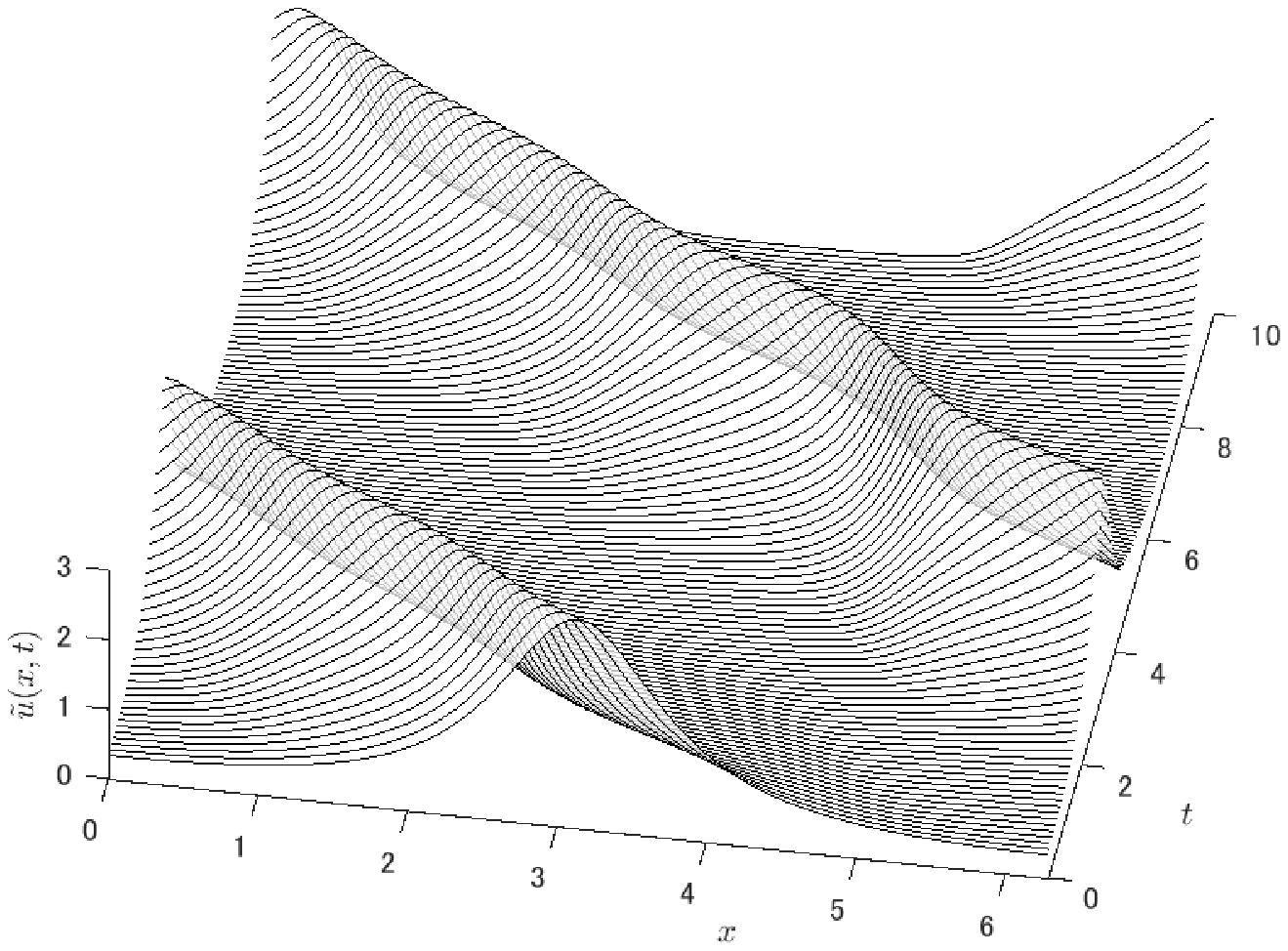}
\caption{Numerically computed approximate solution of \eqref{eqn:advec_Eq} in Example 3.}
\label{fig:fig9}
\end{figure}

Table \ref{Tab:Ex3} lists the \emph{best} numerical result of verified computing when $t_{\max}=0.1,~0.5,~\pi,~2\pi,~4\pi$.
\begin{table}[ht]\em
\caption{Numerical results of verified computing in Example 3.}
\centering
%    \resizebox{\textwidth}{!}{
    \begin{tabular}{rrrrrrrrr}
    \hline
    \multicolumn{1}{c}{$t_{\max}$} & \multicolumn{1}{c}{$N$} & \multicolumn{1}{c}{$n$} & \multicolumn{1}{c}{initial error} & \multicolumn{1}{c}{residual} & \multicolumn{1}{c}{error}  & \multicolumn{1}{c}{app.~time} & \multicolumn{1}{c}{exec.~time}& \multicolumn{1}{c}{ratio}\\
    \hline
%0.1 & 128 & 14 & 5.8775e-39 & 9.2204e-13 & 9.5219e-14 & 1.2675 & 2.0984 & 1.656\\
%0.5 & 87 & 28 & 1.2925e-26 & 9.1739e-13 & 5.4058e-13 & 0.7819 & 0.8364 & 1.070\\
%1.0 & 147 & 49 & 1.121e-44 & 1.9129e-12 & 2.6795e-12 & 3.1211 & 3.1794 & 1.019\\
%$\pi$ & 160 & 127 & 1.3685e-48 & 5.4214e-12 & 5.479e-11 & 10.4269 & 10.5179 & 1.009\\
%$2\pi$ & 163 & 265 & 1.7106e-49 & 1.5839e-11 & 1.3555e-09 & 21.7976 & 21.9415 & 1.007\\
%$4\pi$ & 154 & 536 & 8.7581e-47 & 5.077e-11 & 2.4665e-07 & 39.3953 & 39.6221 & 1.006\\
%0.1 & 128 & 14 & 5.8775e-39 & 9.2204e-13 & 9.5219e-14 & 0.1871 & 0.2178 & 1.164\\
%0.5 & 87 & 28 & 1.2925e-26 & 9.1739e-13 & 5.4058e-13 & 0.6277 & 0.6539 & 1.042\\
%1.0 & 147 & 49 & 1.121e-44 & 1.9129e-12 & 2.6795e-12 & 1.8777 & 1.9169 & 1.021\\
%$\pi$ & 160 & 127 & 1.3685e-48 & 5.4214e-12 & 5.479e-11 & 6.4316 & 6.4871 & 1.009\\
%$2\pi$ & 163 & 265 & 1.7106e-49 & 1.5839e-11 & 1.3555e-09 & 12.7499 & 12.8437 & 1.007\\
%$4\pi$ & 154 & 536 & 8.7581e-47 & 5.0771e-11 & 2.4666e-07 & 25.3205 & 25.4674 & 1.006\\
0.1 & 130 & 14 & 2.2131e-20 & 4.238e-13 & 4.3765e-14 & 0.1897 & 0.2266 & 1.195\\
0.5 & 87 & 28 & 6.5637e-14 & 3.8734e-13 & 3.1864e-13 & 0.5942 & 0.6274 & 1.056\\
1.0 & 200 & 40 & 6.441e-31 & 5.0248e-13 & 7.0385e-13 & 9.3328 & 9.3987 & 1.007\\
$\pi$ & 160 & 127 & 6.7539e-25 & 1.6567e-12 & 1.6743e-11 & 6.6772 & 6.7346 & 1.009\\
$2\pi$ & 197 & 251 & 1.8218e-30 & 4.8822e-12 & 4.1781e-10 & 54.0988 & 54.2070 & 1.002\\
$4\pi$ & 232 & 526 & 9.8282e-36 & 2.3755e-11 & 1.1541e-07 & 156.3342 & 156.5644 & 1.001\\
    \hline
    \end{tabular}%
%    }
\label{Tab:Ex3}
\end{table}
Our method of verified computing also gives tight error bounds between the exact solution and its numerically computed approximate solution.
Moreover, it is executed in high speed.
These are benefit of using the spectral method.

\section*{Conclusion}
In the present paper, we have derived a methodology of verified computing for solutions to 1-dimensional advection equations with variable coefficients based on the Fourier-Chebyshev spectral method.
Main contribution of this paper is to provide a method of verified computing using the $C_0$ semigroup on the complex sequence space $\ell^2$, which comes from the Fourier series of the solution.
We have applied the Lumer-Phillips theorem for an operator $A$ generating the $C_0$ semigroup on the sequence space.
A sufficient condition for the generator of the $C_0$ semigroup is shown in Theorem \ref{thm:sg_generate}.
In Theorem \ref{thm:main_thm}, we have introduced the rigorous error estimate between  the exact solution and its numerically computed approximate solution.
Moreover, numerical results given in Section \ref{sec:num_result} show that the rigorous error estimate is quite accurate and the computational time is as fast as that of usual numerical computations.
%Our method of verified computing essentially corresponds to analysis in the space of wave numbers.
In the field of verified numerical computations, the provided methodology is a foundational approach for initial-boundary value problems of hyperbolic PDEs.

We conclude this paper by discussing some potential extensions.
Our methodology in the present paper could be extended to mathematical models of nonlinear waves.
One example is mathematical models of the traffic stream $\rho_t+(\rho u)_x=0$, where $\rho$ is the traffic density and $u$ denotes vehicle's velocity.
More generally, the provided method could be extended to rigorously compute solutions to multi-dimensional advection equations so-called \emph{Mass transport equation} $u_t+\nabla\cdot(\bm{v}u)=0$ with a given stationary velocity field $\bm{v}$.
Furthermore, a more challenging extension is to rigorously compute solutions of the systems of conservation law $u_t+(f(u))_x=0$ for an unknown function $u(x,t)$.
In all of these extensions, the main difficulty is the loss of smoothness caused by the nonlinearity.
When that happens, the spectral method typically fails and so-called \textit{Gibbs phenomenon} occurs.
On the other hand, there is a vast amount of results (cf. \cite[Chapter 13]{bib:Hesthaven2017}) on the reduction and elimination of the Gibbs phenomenon.
Combining with such techniques, we believe that our semigroup approach still works well in the nonlinear problems.
One interesting future project could be to apply the present method to the \textit{radii-polynomial approach} developed in \cite{bib:Figueras2016,bib:Hungria2016,bib:Lessard2017} etc.\ for rigorous spectral methods in nonlinear hyperbolic PDEs.

%Such extensions to the nonlinear problems require generation theory of the \emph{evolution operator} in sequence spaces.
\vspace{4mm}
%\vskip

\par
\textbf{Acknowledgment.}
The authors express their sincere gratitude to Dr. M. Sobajima in Tyokyo University of Science for his essential suggestion on generating the $C_0$ semigroup on sequence spaces.
%The authors would also like to express their gratitude to the two anonymous referees for providing various useful comments that led to improvements in this paper.
This work was partially supported by JSPS Grant-in-Aid for Early-Career Scientists, No.\ 18K13453.

\bibliographystyle{abbrv}
\bibliography{1dadvection}

\end{document}